\documentclass[10]{amsart}
\usepackage{color}
\usepackage{amsmath}
\usepackage{kotex}
\usepackage{amssymb}
\usepackage{amsthm}
\usepackage{amscd}
\usepackage{bm}
\usepackage{eucal} % for \mathcal
\usepackage{enumitem}
\usepackage{tikz, tikz-cd}
\usepackage[utf8]{inputenc}
\usepackage[all]{xy}
\allowdisplaybreaks

\usepackage[hidelinks]{hyperref}

\theoremstyle{plain}
\newtheorem{theorem}{Theorem}[section]
\newtheorem{conjecture}[theorem]{Conjecture}
\newtheorem{lemma}[theorem]{Lemma}
\newtheorem{proposition}[theorem]{Proposition}
\newtheorem{corollary}[theorem]{Corollary}

\newcommand{\im}{\operatorname{Im}}

\theoremstyle{definition}
\newtheorem{definition}[theorem]{Definition}
\newtheorem{remark}[theorem]{Remark}

\newtheorem{example}[theorem]{Example}

\numberwithin{equation}{section}

\author{Khac Nhuan Le and Kien Huu Nguyen}
\title{On a Conjecture of Gezmis and Pellarin}

\date{\today}


\begin{document}

\maketitle

\begin{abstract}
    In 2022, Gezmis and Pellarin introduced and studied the concept of trivial multiple zeta values, along with a map from the vector space spanned by these values to the vector space spanned by Thakur's multiple zeta values. Their construction allows us to generate some linear relations among the latter values using the former. In our work, we determine the structure of the kernel of the aforementioned map. As a consequence, we give an answer to a conjecture proposed by Gezmis and Pellarin regarding the injectivity of this specific map.
\end{abstract}

\section*{Introduction}

\subsection{Classical multiple zeta values} 
${}$
\par \textit{Multiple zeta values} (MZVs) are real numbers defined as the convergent series
\begin{equation} \label{eq: series representation}
    \zeta(s_1, \dotsc, s_r) = \sum \limits_{n_1 > \cdots > n_r > 0} \dfrac{1}{n_1^{s_1} \cdots n_r^{s_r}},
\end{equation}
where $s_i$ are positive integers and the first component $s_1$ is strictly greater than $1$. These values were first considered by Euler in the 18th century and have been studied in various contexts in number theory, knot theory and the theory of mixed Tate motives. One fundamental problem in the study of MZVs is the identification of all linear relations among them, which remains a challenging open question.

One of the important properties of MZVs is their representation in terms of iterated integral as follows:
\begin{equation} \label{eq: integral representation}
     \zeta(s_1, \dotsc, s_r) =  \int \limits_{1 > t_1 > \cdots > t_k > 0} \omega_1(t_1) \cdots \omega_k(t_k),
\end{equation}
where $k = s_1 + \cdots + s_r$, $\omega_i(t_i) = dt_i/(1 - t_i)$, if $i \in \{s_1, s_1 + s_2, \dotsc, s_1 + \cdots + s_r\}$, and $\omega_i(t_i) = dt_i/t_i$, otherwise. The series representation \eqref{eq: series representation} and the integral representation \eqref{eq: integral representation} provide two different ways of expanding the product of two MZVs as linear combinations of MZVs, resulting in two distinct combinatorial interpretations. The equality of the products then allows us to generate a large family of relations among MZVs called \textit{double shuffle relations}. Nevertheless, these relations are not sufficient to capture all linear relations, for instance, the well-known identity $\zeta(2,1) = \zeta(3)$ due to Euler cannot be derived from them. In order to remedy this, Ihara, Kaneko, and Zagier extended the double shuffle relations by allowing divergent series and integrals, and introduced the so-called \textit{extended double shuffle relations}, which are widely believed to determine all linear relations among MZVs (see \cite[Conjecture 1]{IKZ06}).

%%%%%%%%%%%%%%%%%%%%%

\subsection{Thakur's multiple zeta values}
${}$
\par Let us now consider the function field case. We let $\mathbf{N}$ denote the set of natural numbers, i.e., non-negative integers. The set of positive integers will be denoted by $\mathbf{N}^*$. Let $\mathbf{F}_q$ be a finite field of characteristic $p$ with $q$ elements. We denote by $A = \mathbf{F}_q[\theta]$ the polynomial ring in indeterminate $\theta$ over $\mathbf{F}_q$. Let $A^+$ denote the set of monic polynomials in $A$, and let $A^+(d)$ denote the set of monic polynomials in $A$ whose degree is equal to $d$. We let $K = \mathbf{F}_q(\theta)$ denote the fraction field of $A$, and endow on $K$ a valuation $v_{\infty}$ given by $v_{\infty}(a/b) = \deg b -  \deg a$ for $a/b \in K$, with the convention that $\deg 0 = - \infty$.  Let  $K_{\infty} = \mathbf{F}_q((1/\theta))$ denote the completion of $K$ for this valuation, and let $\mathbf{C}_{\infty}$ denote the completion of an algebraic closure of $K_{\infty}$. 

In \cite{Car35}, Carlitz studied a specific series known as the \textit{Carlitz zeta values}, defined by
\begin{equation*}
    \zeta_A(s) = \sum \limits_{a \in A^+} \dfrac{1}{a^s} \in K_{\infty},
\end{equation*}
where $s$ is a positive integer. By grouping the terms according to the degree $d \in \mathbf{N}$, one can express a Carlitz zeta value as a series of \textit{power sums} given by
\begin{equation*}
    S_d(s) = \sum \limits_{a \in A^+(d)} \dfrac{1}{a^s}.
\end{equation*}
More generally, for a tuple of positive integers $\mathtt{s} = (s_1, \dotsc, s_r)$, Thakur introduced the concept of \textit{multiple zeta values} in the function field setting, as follows:
\begin{equation*}
    \zeta_A(\mathtt{s}) = \sum\limits_{d_1 > \cdots > d_r \geq 0} S_{d_1}(s_1) \cdots S_{d_r}(s_r) = \sum \limits_{\substack{a_i \in A^+\\  \deg a_1> \dotsb > \deg a_r \geq 0}}  \dfrac{1}{a_1^{s_1} \cdots a_r^{s_r}} \in K_{\infty}.
\end{equation*}
When $\mathtt{s} = \emptyset$, we agree by convention that $\zeta_A(\emptyset) = 1$. We call $w= s_1 + \cdots +  s_r$ the \textit{weight} of $\zeta_A(\mathtt{s})$. For $d \in \mathbf{N}$, the \textit{power sum} associated with the tuple $\mathtt{s}$ is defined as 
\begin{equation*}
    S_d(\mathtt{s}) = \sum\limits_{d = d_1 > \cdots > d_r \geq 0} S_{d_1}(s_1) \cdots S_{d_r}(s_r) = \sum \limits_{\substack{a_i \in A^+\\  d = \deg a_1> \dotsb > \deg a_r \geq 0}}  \dfrac{1}{a_1^{s_1} \cdots a_r^{s_r}}.
\end{equation*}
Thakur showed in \cite{Tha10} that the product of two MZVs can be expressed as a $\mathbf{F}_q$-linear combination of MZVs, therefore, the $K$-vector space generated by Thakur's MZVs is a $K$-subalgebra of $K_{\infty}$. Furthermore, Chang showed in \cite{Cha14} that this algebra has a graded structure with respect to weight.

% Although our primary goal, as in the classical case, is to determine all linear relations among MZVs, we face a challenge due to the lack of integral representation for MZVs in function field setting. Therefore, the method of obtaining linear relations via the double shuffle relations is not applicable. However, Todd proposed a conjecture that could provide a solution to this problem. More precisely, he introduced two families of maps between the space of power sum relations and suggested that all linear relations among MZVs can be generated from only one linear relation (see \cite[Cenjecture 5.1]{Tod18}). Using this approach, Todd was able to derive some relations for certain small weights (see \cite[Subsection 5.1 - 5.4]{Tod18}).

%%%%%%%%%%%%%%%%%

\subsection{Gezmis-Pellarin's conjecture}
${}$
\par We now review some works of Gezmis and Pellarin in \cite{GP21}. Let $\{t_i\}_{i \in \mathbf{N}^*}$ and $\{X_i\}_{i \in \mathbf{N}^*}$ be two sequences of independent indeterminates. For any finite subset $\Sigma$ of $\mathbf{N}^*$, we define the character $\chi_{_\Sigma}$  by setting $\chi_{_\Sigma}(a) = \prod_{i \in \Sigma} a(t_i)$ for each polynomial $a \in A$. Here $a(t_i)$ is obtained by substituting $\theta$ with $t_i$ in the polynomial $a$. We also define $X_{\Sigma}$ as the product of $X_i$ for all $i$ in $\Sigma$. Gezmis and Pellarin investigated two series defined as
\begin{align} \label{eq: zeta}
 \zeta_A\begin{pmatrix}
\Sigma_1 & \dotsb & \Sigma_r\\
s_1 & \dotsb & s_r
\end{pmatrix}  &= \sum \limits_{\substack{a_i \in A^+\\  \deg a_1> \dotsb > \deg a_r \geq 0}}  \dfrac{\chi_{_{\Sigma_1}}(a_1) \cdots \chi_{_{\Sigma_r}}(a_r)}{a_1^{s_1} \cdots a_r^{s_r}}, \\
\lambda_A\begin{pmatrix}
\Sigma_1 & \dotsb & \Sigma_r\\
s_1 & \dotsb & s_r
\end{pmatrix} &= \sum \limits_{\substack{a_i \in A^+\\  \deg a_1> \dotsb > \deg a_r \geq 0}}  \dfrac{X_{\Sigma_1}^{q^{\deg a_1}} \cdots X_{\Sigma_r}^{q^{\deg a_r}}}{a_1^{s_1} \cdots a_r^{s_r}}, \label{eq: lambda}
\end{align}
where $\Sigma_1, \dotsc, \Sigma_r$ are disjoint subsets of $\mathbf{N}^*$, and $s_1, \dotsc, s_r$ are positive integers. We call $\Sigma = \Sigma_1 \sqcup \cdots \sqcup \Sigma_r$ the \textit{type} for the above series. The series \eqref{eq: zeta} (resp. \eqref{eq: lambda}) converges to an element of the Tate algebra (see Subsection \ref{subsec: Pellarin's multiple zeta values}) in indeterminates $t_i$ with $i \in \Sigma$ (resp. $X_i$ with $i \in \Sigma$) over $\mathbf{C}_{\infty}$. The first series \eqref{eq: zeta}, known as \textit{Pellarin's multiple zeta values},  was introduced by Pellarin in \cite{Pel16}. The second series \eqref{eq: lambda} can be considered as a variant of the \textit{multiple polylogarithms}. When $\Sigma = \emptyset$, both series coincide and reduce to Thakur's MZVs. Gezmis and Pellarin also extended these series for $\Sigma_1,\dots,\Sigma_r$ are \textit{finite weighted subsets} of $\mathbf{N}^*$, a concept which we will review in Subsection \ref{subsec: finite weighted subsets}.

Let $\Sigma$ be a fixed finite subset of $\mathbf{N}^*$ such that $|\Sigma| < q$. We define $\mathcal{Z}_{\Sigma}(K)$ as the $K$-vector space generated by Pellarin's MZVs of type $\Sigma$ given by \eqref{eq: zeta}, and define $\mathcal{L}_{\Sigma}(K)$ as the $K$-vector space generated by multiple polylogarithms of type $\Sigma$ given by \eqref{eq: lambda}. When $\Sigma = \emptyset$, we have $\mathcal{Z}_{\emptyset}(K) = \mathcal{L}_{\emptyset}(K)$, which is the $K$-algebra generated by Thakur's MZVs. Gezmis and Pellarin showed that both $\mathcal{Z}_{\Sigma}(K)$ and $\mathcal{L}_{\Sigma}(K)$ possess graded $\mathcal{Z}_{\emptyset}(K)$-module structures. They also constructed an isomorphism of graded $\mathcal{Z}_{\emptyset}(K)$-modules $\mathcal{F}_{\Sigma} \colon \mathcal{Z}_{\Sigma}(K) \rightarrow \mathcal{L}_{\Sigma}(K)$, connecting these spaces (see \cite[Theorem 5.2]{GP21}). Furthermore, they constructed a map  $\mathcal{E}_{\Sigma} \colon \mathcal{Z}_{\Sigma}(K) \rightarrow \mathcal{L}_{\Sigma}(K)$ that coincides with $\mathcal{F}_{\Sigma}$ but is defined differently. We shall review the construction of both maps $\mathcal{F}_{\Sigma} $ and $\mathcal{E}_{\Sigma} $ in Subsection \ref{subsec: multiple polylogarithms}.

Gezmis and Pellarin introduced the notion of \textit{trivial multiple zeta values} and considered the graded $\mathcal{Z}_{\emptyset}(K)$-module $\mathcal{Z}^{\text{triv}}_{\Sigma}(K)$ consisting of elements in $\mathcal{Z}_{\Sigma}(K)$ that vanish at $t_i = q^{k_i}$ with $i \in \Sigma$, for all but finitely many tuples $(k_i)_{i \in \Sigma} \in \mathbf{N}^{|\Sigma|}$. For $k \in \mathbf{N}$ and $i \in \mathbf{N}^*$, we define the element
\begin{equation*}
    \eta_k(t_i) = (-1)^k\zeta_A  \begin{pmatrix}
    \{i\} & \emptyset &  \cdots & \emptyset\\ 1 & q-1 &  \cdots &  (q-1)q^{k-1} 
    \end{pmatrix} \in \mathcal{Z}^{\text{triv}}_{\{i\}}(K).
\end{equation*}
We note that the polynomial $\eta_k(t_i)$ used in our work is denoted by $g_k(t_i)$ in \cite{GP21}. Gezmis and Pellarin established a result that characterizes the structure of $\mathcal{Z}_{\emptyset}(K)$-module $\mathcal{Z}^{\text{triv}}_{\Sigma}(K)$ using generators (see \cite[Theorem 6.10]{GP21}). More precisely, the elements $\prod_{i \in \Sigma}\eta_{k_i}(t_i)$ with $k_i \in \mathbf{N}$ for all $i \in \Sigma$ form a generating set of the $\mathcal{Z}_{\emptyset}(K)$-module $\mathcal{Z}^{\text{triv}}_{\Sigma}(K)$.
% The following result of Gezmis and Pellarin characterizes the structure of $\mathcal{Z}_{\emptyset}(K)$-module $\mathcal{Z}^{\text{triv}}_{\Sigma}(K)$ using generators (see \cite[Theorem 6.10]{GP21}).
% \begin{theorem}
%     Let $\Sigma$ be a finite subset of $\mathbf{N}^*$ such that $|\Sigma| < q$. Those elements $\prod_{i \in \Sigma}\eta_{k_i}(t_i)$ with $k_i \in \mathbf{N}$ for all $i \in \Sigma$ form a generating set of the $\mathcal{Z}_{\emptyset}(K)$-module $\mathcal{Z}^{\text{triv}}_{\Sigma}(K)$.
% \end{theorem}

We define a map $\mathcal{G}_{\Sigma}$ as the composition of the following maps:
\begin{center}
    \begin{tikzcd}
\mathcal{Z}^{\text{triv}}_{\Sigma}(K) \arrow[r, "\mathcal{F}_{\Sigma}", shift left] \arrow[r, "\mathcal{E}_{\Sigma}"', shift right] & \mathcal{L}_{\Sigma}(K) \arrow[r, "{\text{ev}}"] & \mathcal{Z}_{\emptyset}(K),
\end{tikzcd}
\end{center}
where the first map, $\mathcal{F}_{\Sigma} = \mathcal{E}_{\Sigma}$, is restricted on $\mathcal{Z}^{\text{triv}}_{\Sigma}(K)$, and the second map $\text{ev}$ is the evaluation map at $X_i = 1$ for all $i \in \Sigma$. Using $\mathcal{F}_{\Sigma}$ and $\mathcal{E}_{\Sigma}$, a trivial multiple zeta value can be mapped to linear combinations of multiple polylogarithms in two distinct combinatorial ways. The equality of the two maps generates a family of relations among multiple polylogarithms. Evaluating at $X_i = 1$ for all $i \in \Sigma$ yields a family of relations among Thakur's MZVs. This approach can be considered as a partial alternative for the double shuffle relations. Gezmis and Pellarin used this approach to recover some previously known linear relations among MZVs due to Lara Rodríguez and Thakur, as well as to derive a new family of linear relations (see \cite[Subsection 7.1]{GP21}). Regarding the map $\mathcal{G}_{\Sigma}$ defined above, Gezmis and Pellarin proposed the following conjecture (see \cite[Conjecture 6.15]{GP21}):
\begin{conjecture} \label{conj: injectivity of G_Sigma}
The map $\mathcal{G}_{\Sigma} \colon \mathcal{Z}^{\text{triv}}_{\Sigma}(K) \rightarrow \mathcal{Z}_{\emptyset}(K)$ is injective.
\end{conjecture}
 % Assuming Conjecture \ref{conj: injectivity of G_Sigma} holds, it follows that any linear relation among Thakur's MZVs in the image of $\mathcal{G}_{\Sigma}$ can be obtained from linear relations among trivial multiple zeta values. Although the entire space $\mathcal{Z}_{\emptyset}(K)$ may not be covered by the image of $\mathcal{G}_{\Sigma}$, Gezmis and Pellarin demonstrated that their approach is effective in recovering some relations among Thakur's MZVs of low weight that were previously derived by Todd (see \cite[Subsection 7.2]{GP21}).

%%%%%%%%%%%%%%%%%%

\subsection{Results and outline}
${}$
\par 

Set $D_0 = 1$ and $D_k = \prod^{k-1}_{i=0}(\theta^{q^k} - \theta^{q^i})$ for all $k \in \mathbf{N}^*$. The main result of our work characterizes the structure of the kernel of the map $\mathcal{G}_{\Sigma}$. More precisely, it reads as follows:
\begin{theorem} \label{thm: main A}
    Let $\Sigma$ be a finite subset of $\mathbf{N}^*$ such that $|\Sigma| < q$. Then we have $\ker(\mathcal{G}_{\Sigma})$ is a free $\mathcal{Z}_{\emptyset}(K)$-module. In particular, there exists a $\mathcal{Z}_{\emptyset}(K)$-basis of $\ker(\mathcal{G}_{\Sigma})$ consisting of elements
    \begin{equation*}
        - \zeta_A(1)^{\sum_{i \in \Sigma}q^{k_i} - 1}\prod\limits_{i \in \Sigma} \eta_0(t_i) + \prod\limits_{i \in \Sigma} D_{k_i}\eta_{k_i}(t_i)
    \end{equation*}
with $k_i \in \mathbf{{N}}$ and $k_i$ are not all equal to zero for $i \in \Sigma$.
\end{theorem}
To prove Theorem \ref{thm: main A}, we use the structure of $\mathcal{Z}_{\emptyset}(K)$-module $\mathcal{Z}^{\text{triv}}_{\Sigma}(K)$ (see \cite[Theorem 6.10]{GP21}) to interpolate elements of $\ker(\mathcal{G}_{\Sigma})$ at points $(\theta^{q^{k_i}})_{i \in \Sigma}$
% , using the elements $\prod_{i \in \Sigma}\eta_{k_i}(t_i)$ 
with $k_i \in \mathbf{N}$ for all $i \in \Sigma$. The construction of the map $\mathcal{E}_{\Sigma}$ then allows us to establish the algebraic structure of $\ker(\mathcal{G}_{\Sigma})$. The linear independence of the basis follows from the observation that $\eta_k(q^i) = 1$, if $i = k$ and  $\eta_k(q^i) = 0$, otherwise.

As a consequence of Theorem \ref{thm: main A}, we obtain an answer to Conjecture \ref{conj: injectivity of G_Sigma}.
\begin{corollary}\label{ninjective}
When $\Sigma \ne \emptyset$, the map $\mathcal{G}_{\Sigma}$ is not injective.
\end{corollary}
On the other hand, by using Theorem \ref{thm: main A}, we can also compute the image of $\mathcal{G}_{\Sigma}$.
\begin{theorem}\label{thm: main B}Let $\Sigma$ be a finite subset of $\mathbf{N}^*$ such that $|\Sigma| < q$. Then $\im(\mathcal{G}_{\Sigma})$ is the ideal of $\mathcal{Z}_{\emptyset}(K)$ generated by $\zeta_A(1)^{|\Sigma|}$.
\end{theorem}

Let us give an outline of the paper. In Section \ref{sec: pellarin's multiple zeta values},  we first review some basic concepts and introduce Pellarin's multiple zeta values. We then establish the sum-shuffle formula for the products of two zeta values. In Section \ref{sec: a formula for power sums}, we derive a formula for power sum that corrects the original formula proposed by Gezmis and Pellarin \cite[Formula (22)]{GP21}. We then verify and compare these two formulas and reprove \cite[Corollary 5.4]{GP21}. In Section \ref{sec: counterexamples}, we review the construction of the maps $\mathcal{F}_{\Sigma}$ and $\mathcal{E}_{\Sigma}$, as well as the concept of trivial multiple zeta values. Then we provide the proofs of Theorem \ref{thm: main A}, Corollary \ref{ninjective} and Theorem  \ref{thm: main B}.

%%%%%%%%%%%%%%%%%%%

\subsection*{Acknowledgments}
We are grateful to T. Ngo Dac for his encouragement and his suggestions throughout this work. We also want to thank F. Pellarin, who shared his insights and patiently answered all of our questions.

Two authors are supported by the Excellence Research Chair “L-functions in positive characteristic and applications” financed by the Normandy Region.

%%%%%%%%%%%%%%%%%%%%

\section{Pellarin's multiple zeta values} \label{sec: pellarin's multiple zeta values}

\subsection{Notations and conventions}
${}$
\par Let $p$ be a prime number and let $q$ be a power of $p$. In this work, we will use constantly the following notations:
\begin{align*}
    & \mathbf{N} = \text{set of natural numbers, i.e., non-negative integers.}\\
    & \mathbf{N}^* = \text{set of positive integers.} \\
    & \mathbf{Z} = \text{set of integers.} \\
    & \mathbf{F}_q = \text{finite field with $q$ elements.} \\
    & A = \text{polynomial ring $\mathbf{F}_q[\theta]$ in indeterminate $\theta$ over $\mathbf{F}_q$.} \\
    & A^+ = \text{set of monic polynomials in $A$.} \\
    & A^+(d) = \text{set of monic polynomials in $A$ whose degree is equal to $d$.} \\
    & A_{<}(d) = \text{set of polynomials in $A$ whose degree is strictly less than $d$.} \\
    & K = \text{fraction field $\mathbf{F}_q(\theta)$ of $A$.} \\
    & v_{\infty} = \text{valuation on } K \text{ given by } v_{\infty}(a/b) = \deg b - \deg a \text{ for } a/b \in K.\\
    & K_{\infty} = \text{completion $\mathbf{F}_q((1/\theta))$ of $K$ for $v_{\infty}$.} \\
    & \mathbf{C}_{\infty} = \text{completion of an algebraic closure of $K_{\infty}$.} \\
    & \ell_d = \prod^d_{i=1}(\theta - \theta^{q^i}) \text{ for } d \in \mathbf{N}. \\
    & D_d = \prod^{d-1}_{i=0}(\theta^{q^d} - \theta^{q^i}) \text{ for } d \in \mathbf{N}.\\
    &  b_d(t) = \prod_{i =0}^{d-1}(t-\theta^{q^i}) \text{ for } d \in \mathbf{N}.
\end{align*}
As a matter of convention, we agree that an empty product is equal to $1$ and an empty sum is equal to $0$.

%%%%%%%%%%%%%%%%%%%

\subsection{Finite weighted subsets} \label{subsec: finite weighted subsets}
${}$
 \par A countable set $\Sigma = \{\sigma_n\}_{n \in \mathbf{N}^*}$ of natural numbers is called a \textit{finite weighted subset} of $\mathbf{N}^*$ if $\sigma_n = 0$ for all but finitely many $n \in \mathbf{N}^*$. The \textit{support} $\text{Supp}(\Sigma)$ of $\Sigma$ is the set of natural numbers $n \in \mathbf{N}^*$ such that $\sigma_n \ne 0$. The \textit{cardinality} of $\Sigma$, denoted by $|\Sigma|$, is defined as the sum
 \begin{equation*}
        |\Sigma| = \sum_{n \in \mathbf{N}^*} \sigma_n.
    \end{equation*}
In particular, we write $\emptyset$ (by abuse of notation) for the finite weighted subset of $\mathbf{N}^*$ whose elements are all zeros.

\begin{remark}\label{rmk: finite subsets as fws}
One can regard a usual finite subset $\Sigma$ of $\mathbf{N}^*$ as a finite weighted subset $\Sigma = \{\sigma_n\}_{n \in \mathbf{N}^*}$ of $\mathbf{N}^*$ with $\sigma_n = 1$ if $n \in \Sigma$ and $\sigma_n = 0$ if $n \notin \Sigma$.
\end{remark}

Let $\Sigma = \{\sigma_n\}_{n \in \mathbf{N}^*}$ be a finite weighted subset of $\mathbf{N}^*$. A finite weighted subset $J = \{j_n\}_{n \in \mathbf{N}^*}$ of $\mathbf{N}^*$ is called a \textit{subset} of $\Sigma$ if $j_n \leq \sigma_n$ for all $n \in \mathbf{N}^*$. If this is the case, we write $J \subseteq \Sigma$. 

Let $\Sigma = \{\sigma_n\}_{n \in \mathbf{N}^*}$ and $\Gamma = \{\gamma_n\}_{n \in \mathbf{N}^*}$ be two finite weighted subsets of $\mathbf{N}^*$. We define the \textit{union} of $\Sigma$ and $\Gamma$, denoted by $\Sigma \sqcup \Gamma$, as the finite weighted subset of $\mathbf{N}^*$ given by 
\begin{equation*}
 \Sigma \sqcup \Gamma = \{\sigma_n + \gamma_n\}_{n \in \mathbf{N}^*}.
\end{equation*}

%%%%%%%%%%%%%%%

\subsection{Pellarin's multiple zeta values} \label{subsec: Pellarin's multiple zeta values}
${}$
 \par Let us first review the concept of \textit{Tate algebras}. We consider a finite number of indeterminates $t_1, \dotsc, t_s$ over $\mathbf{C}_{\infty}$. We endow on $\mathbf{C}_{\infty}[t_1, \dotsc, t_s]$ the \textit{Gauss valuation} $v_{\infty}$, defined for any polynomial $f = \sum a_{n_1, \dotsc, n_s}t_1^{n_1} \cdots t_s^{n_s} \in \mathbf{C}_{\infty}[t_1, \dotsc, t_s]$ by
    \begin{equation*}
        v_{\infty}(f) = \min \limits_{n_1, \dotsc, n_s} v_{\infty}(a_{n_1, \dotsc, n_s}).
    \end{equation*}
    Denote by $\mathbf{T}_s(\mathbf{C}_{\infty})$ the completion of $\mathbf{C}_{\infty}[t_1, \dotsc, t_s]$ with respect to the Gauss valuation. 
    Then $\mathbf{T}_s(\mathbf{C}_{\infty})$ is a $\mathbf{C}_{\infty}$-algebra, and is known as a  \textit{Tate algebra} over $\mathbf{C}_{\infty}$ in indeterminates $t_1, \dotsc, t_s$. Moreover, we can identify $\mathbf{T}_s(\mathbf{C}_{\infty})$ with the set of formal power series $f = \sum a_{n_1, \dotsc, n_s}t_1^{n_1} \cdots t_s^{n_s} \in \mathbf{C}_{\infty}[[t_1, \dotsc, t_s]]$ such that 
    \begin{equation*}
        \lim\limits_{n_1 + \cdots + n_s \rightarrow \infty} a_{n_1, \dotsc, n_s} = 0.
    \end{equation*}
    We refer the reader to \cite[Chapter 5]{BGR84} for further properties of Tate algebras.
 
 Let $\{t_n\}_{n \in \mathbf{N}^*}$ be a sequence of independent indeterminates. For any finite weighted subset $\Sigma = \{\sigma_n\}_{n \in \mathbf{N}^*}$ of $\mathbf{N}^*$, we let $\mathbf{T}_{\Sigma}(\mathbf{C}_{\infty})$ denote the \textit{Tate algebra}  over $\mathbf{C}_{\infty}$ in indeterminates $t_n$ with $n \in \text{Supp}(\Sigma)$.  We define the character $\chi_{_\Sigma}$ as follows: for each polynomial $a$ in $A$, we set
\begin{equation*}
    \chi_{_\Sigma} (a) = \prod \limits_{n \in \mathbf{N}^*} a(t_n)^{\sigma_n},
\end{equation*}
where $a(t_n)$ is the polynomial obtained by substituting $\theta$ with $t_n$ in $a$. When $\Sigma = \{i\}$ is a singleton set, we write simply $\chi_{t_i} (a)$ instead of $\chi_{_\Sigma} (a)$. 

Let $\Sigma = \{\sigma_n\}_{n \in \mathbf{N}^*}$ be a finite weighted subset of $\mathbf{N}^*$, and let $J = \{j_n\}_{n \in \mathbf{N}^*}$ be a subset of $\Sigma$. We define
\begin{align*}
    \binom{\Sigma}{J} = \prod \limits_{n \in \mathbf{N}^*}\binom{\sigma_n}{j_n}.
\end{align*}
The following proposition gives some key properties of the character $\chi_{_\Sigma}$, which are used frequently later.
\begin{proposition} \label{prop: character}
    We have the following properties:
\begin{enumerate}[label=\textnormal{(\roman*)}, align= left]
\item $\chi_{_\Sigma}(a+b) = \sum \limits_{I \sqcup J = \Sigma} \binom{\Sigma}{J}\chi_{_I}(a) \chi_{_J}(b)$.
\item  $\chi_{_\Sigma}(ab) = \chi_{_\Sigma}(a)\chi_{_\Sigma}(b)$.
\item  For $\alpha \in \mathbf{F}_q$, $\chi_{_\Sigma}(\alpha a) = \alpha^{|\Sigma|}\chi_{_\Sigma}(a)$.
\item  $\chi_{_\Sigma}(a)\chi_{_\Gamma}(a) = \chi_{_{\Sigma \sqcup \Gamma}}(a)$.
\end{enumerate}
\end{proposition}

\begin{proof}
    The properties (ii), (iii), (iv) are trivial. To prove (i), we assume that $\Sigma = \{\sigma_n\}_{n \in \mathbf{N}^*}$ is a finite weighted subset $\Sigma$ of $\mathbf{N}^*$. Then we have
    \begin{align*}
        \chi_{_\Sigma}(a+b) = \prod \limits_{n \in \mathbf{N}^*} (a(t_n) + b(t_n))^{\sigma_n} =& \prod\limits_{n \in \mathbf{N}^*} \sum\limits_{i_n + j_n = \sigma_n}  \binom{\sigma_n}{j_n}a(t_n)^{i_n}b(t_n)^{j_n} \\
        =& \sum \limits_{I \sqcup J = \Sigma} \binom{\Sigma}{J}\chi_{_I}(a) \chi_{_J}(b).
    \end{align*}
\end{proof}

% Let $\Sigma$ be a finite weighted subset of $\mathbf{N}^*$, and let $s$ be a positive integer. For $d \in \mathbf{N}$, we define
% \begin{align*}
%     S_d\begin{pmatrix}
% \Sigma\\
% s
% \end{pmatrix}  = \sum\limits_{a \in A^+(d)}  \dfrac{\chi_{_\Sigma}(a)}{a^s}.
% \end{align*}
Let $\Sigma_1, \dotsc , \Sigma_r$ be finite weighted subsets of $\mathbf{N}^*$, and let $s_1, \dotsc, s_r$ be integers. An array $\mathtt{A} =  \begin{pmatrix}
\Sigma_1 & \dotsb & \Sigma_r\\
s_1 & \dotsb & s_r
\end{pmatrix}$ is said to be \textit{admissible} if $s_1, \dotsc, s_r$ are positive integers. For such an admissible array, we call $\mathtt{A}$ a \textit{composition} of $\begin{pmatrix}
\Sigma\\
w
\end{pmatrix}$ if $\Sigma_1 \sqcup \cdots \sqcup \Sigma_r = \Sigma$ and $s_1 + \cdots + s_r = w$. Let $\mathtt{A} =  \begin{pmatrix}
\Sigma_1 & \dotsb & \Sigma_r\\
s_1 & \dotsb & s_r
\end{pmatrix}$ be an admissible array which is a composition of $\begin{pmatrix}
\Sigma\\
w
\end{pmatrix}$. The \textit{Pellarin's multiple zeta value} (Pellarin' MZV for short) associated with $\mathtt{A}$ is defined as
\begin{align*}
 \zeta_A(\mathtt{A}) =  \sum \limits_{\substack{a_i \in A^+\\  \deg a_1> \dotsb > \deg a_r \geq 0}}  \dfrac{\chi_{_{\Sigma_1}}(a_1) \cdots \chi_{_{\Sigma_r}}(a_r)}{a_1^{s_1} \cdots a_r^{s_r}} \in \mathbf{T}_{\Sigma}(\mathbf{C}_{\infty}).
%  \sum\limits_{d_1 > \cdots > d_r \geq 0} S_{d_1}\begin{pmatrix}
% \Sigma_1\\
% s_1
% \end{pmatrix} \cdots S_{d_r}\begin{pmatrix}
% \Sigma_r\\
% s_r
% \end{pmatrix} =
\end{align*}
We call $\Sigma$ the \textit{type}, $w$ the \textit{weight} and $r$ the \textit{depth} of $\mathtt{A}$ (resp. $\zeta_A(\mathtt{A})$).  It was shown by Pellarin in \cite[Proposition 4]{Pel16} that every Pellarin's MZV can be considered as an entire function $\mathbf{C}_{\infty}^{|\text{Supp}(\Sigma)|} \rightarrow \mathbf{C}_{\infty}$. For $d \in \mathbf{N}$, the \textit{power sum} associated with $\mathtt{A}$ is defined as
\begin{align*}
    S_d(\mathtt{A}) =  \sum \limits_{\substack{a_i \in A^+\\  d = \deg a_1> \dotsb > \deg a_r \geq 0}}  \dfrac{\chi_{_{\Sigma_1}}(a_1) \cdots \chi_{_{\Sigma_r}}(a_r)}{a_1^{s_1} \cdots a_r^{s_r}}.
%  = \sum\limits_{d = d_1 > \cdots > d_r \geq 0} S_{d_1}\begin{pmatrix}
% \Sigma_1\\
% s_1
% \end{pmatrix} \cdots S_{d_r}\begin{pmatrix}
% \Sigma_r\\
% s_r
% \end{pmatrix}.
\end{align*}
Moreover, for $d \in \mathbf{N}$, we define 
\begin{align*}
    S_{<d}(\mathtt{A}) = \sum \limits_{\substack{a_i \in A^+\\  d > \deg a_1> \dotsb > \deg a_r \geq 0}}  \dfrac{\chi_{_{\Sigma_1}}(a_1) \cdots \chi_{_{\Sigma_r}}(a_r)}{a_1^{s_1} \cdots a_r^{s_r}}.
%     = \sum\limits_{d > d_1 > \cdots > d_r \geq 0} S_{d_1}\begin{pmatrix}
% \Sigma_1\\
% s_1
% \end{pmatrix} \cdots S_{d_r}\begin{pmatrix}
% \Sigma_r\\
% s_r
% \end{pmatrix}.
\end{align*}
Thus we have
\begin{align*}
S_{<d}(\mathtt{A}) &= \sum\limits_{k = 0}^{d-1} S_k(\mathtt{A}) \quad \text{and} \quad 
    \zeta_A(\mathtt{A}) = \sum\limits_{d = 0}^{\infty} S_d(\mathtt{A}).
\end{align*}
One verifies easily that 
\begin{align}
S_d \begin{pmatrix}
 \emptyset& \dotsb & \emptyset \\
s_1 & \dotsb & s_r \end{pmatrix}  &= S_{d}(s_1, \dotsc, s_r),\label{eq: power sum formula 1}\\
S_{d} \begin{pmatrix}
\Sigma_1 & \dotsb & \Sigma_r\\
s_1 & \dotsb & s_r
\end{pmatrix}  &= S_{d} \begin{pmatrix}
\Sigma_1\\
s_1 
\end{pmatrix}S_{<d} \begin{pmatrix}
\Sigma_2 & \dotsb & \Sigma_r\\
s_2 & \dotsb & s_r
\end{pmatrix} \label{eq: power sum formula 3}.
\end{align}
  
%%%%%%%%%%%%

\subsection{A sum-shuffle formula for power sums of depth $1$} 
${}$
 \par In this subsection, we establish an explicit sum-shuffle formula for the products of two zeta values (see Theorem \ref{thm: sum-shuffle formula}). We refer the reader to Pellarin's paper \cite{Pel17} for another approach to this formula.
 
 Let $a, b$ be two integers with $b \geq 0$. We recall the \textit{binomial number}
\begin{align*}
\binom{a}{b} = \dfrac{a(a-1) \cdots (a-b+1)}{b!}.
\end{align*}
It should be remarked that $\binom{a}{0} = 1$ and $\binom{a}{b} = 0$ when $b > a \geq 0$.

\begin{lemma} \label{lem: basic indentity}
Let $s,t$ be two positive integers. The following equality of rational function holds:
\begin{equation*}
    \frac{1}{X^s Y^t}=\sum\limits_{i+j=s+t} \left[\frac{{j-1 \choose t-1}}{X^i(X+Y)^j}+\frac{{j-1 \choose s-1}}{ Y^i(X+Y)^j}\right],
\end{equation*}
where the indices $i,j$ run through all positive integers.
\end{lemma}

\begin{proof}
See \cite[Lemma 1.49]{BGF}.
\end{proof}

Let $\Sigma$ be a finite weighted subset of $\mathbf{N}^*$, and let $J$ be a subset of $\Sigma$.  Let $j,s$ be two positive integers. For the convenience of writing, we define
\begin{equation*}
    \delta^{\Sigma,j}_{J,s} = (-1)^{|J|+s} \binom{\Sigma}{J}{j-1 \choose s-1}.
\end{equation*}
\begin{lemma} \label{lem: sum-shuffle formula}
Let $\Sigma,\Gamma$ be two finite weighted subsets of $\mathbf{N}^*$, and let $s,t$ be two positive integers. For any $a,b \in A$ such that $a \ne b$, we have
\begin{align*}
    \frac{\chi_{_\Sigma}(a)\chi_{_\Gamma}(b)}{a^s b^t} = \sum \limits_{\substack{J \subseteq \Gamma ; I \sqcup J = \Sigma \sqcup \Gamma \\
i + j = s + t}} \frac{\delta^{\Gamma,j}_{J,t}\chi_{_I}(a)\chi_{_J}(a-b)}{a^i(a-b)^j} + \sum \limits_{\substack{J \subseteq \Sigma ; I \sqcup J = \Sigma \sqcup \Gamma \\
i + j = s + t}} \frac{\delta^{\Sigma,j}_{J,s}\chi_{_I}(b)\chi_{_J}(b-a)}{b^i(b-a)^j},
\end{align*}
where the indices $i,j$ run through all positive integers.
\end{lemma}

\begin{proof}
Replacing $X = a$ and $Y = - b$, it follows from Lemma \ref{lem: basic indentity} that 
\begin{equation*} 
\frac{1}{a^s b^t}=\sum_{i+j= s + t} \left[\frac{(-1)^{t} {j-1 \choose t-1}}{a^i(a-b)^j }+\frac{(-1)^{s} {j-1 \choose s-1}}{b^i(b-a)^j }\right].
\end{equation*}
Thus it follows from Proposition \ref{prop: character} that
\begin{align*}
     &\frac{\chi_{_\Sigma}(a)\chi_{_\Gamma}(b)}{a^s b^t}\\
     =&\sum \limits_{i+j= s + t} \Bigg(\frac{(-1)^{t} {j-1 \choose t-1} \chi_{_\Sigma}(a)\chi_{_\Gamma}(b)}{a^i(a-b)^j}
        + 
        \frac{(-1)^{s} {j-1 \choose s-1} \chi_{_\Sigma}(a)\chi_{_\Gamma}(b)}{b^i(b-a)^j}\Bigg)\\
      =&\sum \limits_{i+j= s + t} \Bigg(\frac{(-1)^{t} {j-1 \choose t-1} \chi_{_\Sigma}(a)\chi_{_\Gamma}(a + b - a)}{a^i(a-b)^j}
        + 
        \frac{(-1)^{s} {j-1 \choose s-1} \chi_{_\Gamma}(b)\chi_{_\Sigma}(b + a - b)}{b^i(b-a)^j}\Bigg)\\
        =&\sum \limits_{i+j= s + t} \Bigg(\frac{(-1)^{t} {j-1 \choose t-1} \chi_{_\Sigma}(a)\sum \limits_{I \sqcup J = \Gamma}\binom{\Gamma}{J}\chi_{_I}(a)\chi_{_J}(b - a)}{a^i(a-b)^j}
        + 
        \frac{(-1)^{s} {j-1 \choose s-1} \chi_{_\Gamma}(b)\sum \limits_{I \sqcup J = \Sigma}\binom{\Sigma}{J}\chi_{_I}(b)\chi_{_J}(a-b)}{b^i(b-a)^j}\Bigg)\\
        =&\sum \limits_{\substack{I \sqcup J = \Gamma\\ i+j= s + t}} \frac{(-1)^{t} \binom{\Gamma}{J}{j-1 \choose t-1} \chi_{_{\Sigma \sqcup I}}(a)\chi_{_J}(b - a)}{a^i(a-b)^j}
        + 
        \sum \limits_{\substack{I \sqcup J = \Sigma\\ i+j= s + t}}\frac{(-1)^{s}\binom{\Sigma}{J} {j-1 \choose s-1} \chi_{_{\Gamma \sqcup I}}(b)\chi_{_J}(a-b)}{b^i(b-a)^j}\\
        =&\sum \limits_{\substack{I \sqcup J = \Gamma\\ i+j= s + t}} \frac{(-1)^{|J| + t} \binom{\Gamma}{J}{j-1 \choose t-1} \chi_{_{\Sigma \sqcup I}}(a)\chi_{_J}(a-b)}{a^i(a-b)^j}
        + 
        \sum \limits_{\substack{I \sqcup J = \Sigma\\ i+j= s + t}}\frac{(-1)^{|J| + s}\binom{\Sigma}{J} {j-1 \choose s-1} \chi_{_{\Gamma \sqcup I}}(b)\chi_{_J}(b-a)}{b^i(b-a)^j}\\
        =& \sum \limits_{\substack{J \subseteq \Gamma ; I \sqcup J = \Sigma \sqcup \Gamma \\
i + j = s + t}} \frac{\delta^{\Gamma,j}_{J,t}\chi_{_I}(a)\chi_{_J}(a-b)}{a^i(a-b)^j} + \sum \limits_{\substack{J \subseteq \Sigma ; I \sqcup J = \Sigma \sqcup \Gamma \\
i + j = s + t}} \frac{\delta^{\Sigma,j}_{J,s}\chi_{_I}(b)\chi_{_J}(b-a)}{b^i(b-a)^j}.
\end{align*}
This proves the lemma.
\end{proof}

Let $\Sigma$ be a finite weighted subset of $\mathbf{N}^*$, and let $J$ be a subset of $\Sigma$.  Let $j,s$ be two positive integers. We define
\begin{align*}
\Delta^{\Sigma,j}_{J,s}=\begin{cases} (-1)^{|J| + s -1}\binom{\Sigma}{J}\binom{j - 1}{s - 1}, & \quad \text{if } |J| \equiv j \ (\text{mod } q - 1), \\
0, & \quad \text{otherwise}.
\end{cases}
\end{align*}

In the following result, we adapt some techniques from Chen's proof regarding the sum-shuffle formula for the product of Carlitz zeta values (see \cite[Theorem 3.1]{Che15}).

\begin{proposition} \label{prop: sum-shuffle formula}
Let $\Sigma,\Gamma$ be two finite weighted subsets of $\mathbf{N}^*$, and let $s,t$ be two positive integers. For all $d \in \mathbf{N}$, we have  
\begin{align*}
    S_d & \begin{pmatrix}
 \Sigma \\
s  \end{pmatrix}  S_d \begin{pmatrix}
 \Gamma \\
t  \end{pmatrix} = S_d \begin{pmatrix}
 \Sigma \sqcup \Gamma \\
s + t  \end{pmatrix} + \sum \limits_{\substack{J \subseteq \Gamma ; I \sqcup J = \Sigma \sqcup \Gamma \\
i + j = s + t}} \Delta^{\Gamma,j}_{J,t} S_d \begin{pmatrix}
 I  &   J\\
i & j  \end{pmatrix} + \sum \limits_{\substack{J \subseteq \Sigma ; I \sqcup J = \Sigma \sqcup \Gamma\\
i + j = s + t}} \Delta^{\Sigma,j}_{J,s} S_d \begin{pmatrix}
 I  &   J\\
i & j  \end{pmatrix},
\end{align*}
where the indices $i,j$ run through all positive integers.
\end{proposition}

\begin{proof}
We have 
    \begin{align*}
        S_d\begin{pmatrix}
        \Sigma\\s
        \end{pmatrix} S_d\begin{pmatrix}
        \Gamma\\t
        \end{pmatrix} &= \left( \sum \limits_{a \in A^+(d)} \dfrac{\chi_{_\Sigma}(a)}{a^s} \right) \left( \sum \limits_{b \in A^+(d)} \dfrac{\chi_{_\Gamma}(b)}{b^t} \right)\\
        &= \sum \limits_{\substack{a,b \in A^+(d) \\ a = b}} \dfrac{\chi_{_\Sigma}(a)\chi_{_\Gamma}(b)}{a^sb^t} + \sum \limits_{\substack{a,b \in A^+(d) \\ a \ne b}} \dfrac{\chi_{_\Sigma}(a)\chi_{_\Gamma}(b)}{a^sb^t}\\
        &= \sum \limits_{a \in A^+(d)} \dfrac{\chi_{_{\Sigma \sqcup \Gamma}}(a)}{a^{s+t}} + \sum \limits_{\substack{a,b \in A^+(d) \\ a \ne b}} \dfrac{\chi_{_\Sigma}(a)\chi_{_\Gamma}(b)}{a^sb^t} \quad \text{(from Proposition \ref{prop: character}(iv))}\\
        &= S_d\begin{pmatrix}
        \Sigma \sqcup \Gamma\\ s+t
        \end{pmatrix} + \sum \limits_{\substack{a,b \in A^+(d) \\ a \ne b}} \dfrac{\chi_{_\Sigma}(a)\chi_{_\Gamma}(b)}{a^sb^t}.
    \end{align*}
    Note that if $a,b$ are two polynomials in $A^+(d)$ such that $a \ne b$, then there exists a unique element $\alpha \in \mathbf{F}_q^{\times}$ and a unique polynomial $c \in A^+$ with $\deg c < d$ such that $a - b = \alpha c$. One deduces from Lemma \ref{lem: sum-shuffle formula} and Proposition \ref{prop: character}(iii) that
    \begin{align*}
        &\sum \limits_{\substack{a,b \in A^+(d) \\ a \ne b}} \dfrac{\chi_{_\Sigma}(a)\chi_{_\Gamma}(b)}{a^sb^t}\\
        & = \sum \limits_{\substack{J \subseteq \Gamma ; I \sqcup J = \Sigma \sqcup \Gamma \\
i + j = s + t}}\sum \limits_{\substack{a,b \in A^+(d) \\ a \ne b}} \frac{\delta^{\Gamma,j}_{J,t}\chi_{_I}(a)\chi_{_J}(a-b)}{a^i(a-b)^j} + \sum \limits_{\substack{J \subseteq \Sigma ; I \sqcup J = \Sigma \sqcup \Gamma \\
i + j = s + t}}\sum \limits_{\substack{a,b \in A^+(d) \\ a \ne b}} \frac{\delta^{\Sigma,j}_{J,s}\chi_{_I}(b)\chi_{_J}(b-a)}{b^i(b-a)^j}\\
        & = \sum \limits_{\substack{J \subseteq \Gamma ; I \sqcup J = \Sigma \sqcup \Gamma \\
i + j = s + t}}\sum \limits_{\substack{\alpha \in \mathbf{F}_q^{\times}; a,c \in A^+ \\ d = \deg a > \deg c}} \frac{\delta^{\Gamma,j}_{J,t}\chi_{_I}(a)\chi_{_J}(\alpha c)}{a^i(\alpha c)^j} + \sum \limits_{\substack{J \subseteq \Sigma ; I \sqcup J = \Sigma \sqcup \Gamma \\
i + j = s + t}}\sum \limits_{\substack{\alpha \in \mathbf{F}_q^{\times}; b,c \in A^+ \\ d = \deg b > \deg c}} \frac{\delta^{\Sigma,j}_{J,s}\chi_{_I}(b)\chi_{_J}(\alpha c)}{b^i(\alpha c)^j}\\
%& = \sum \limits_{\substack{J \subseteq \Gamma ; I \sqcup J = \Sigma \sqcup \Gamma \\
%i + j = s + t}}\sum \limits_{\substack{\alpha \in \mathbf{F}_q^{\times}; a,c \in A^+ \\ d = \deg a > \deg c}} \frac{\alpha^{|J| - j}\delta^{\Gamma,j}_{J,t}\chi_{_I}(a)\chi_{_J}(c)}{a^ic^j} + \sum \limits_{\substack{J \subseteq \Sigma ; I \sqcup J = \Sigma \sqcup \Gamma \\
%i + j = s + t}}\sum \limits_{\substack{\alpha \in \mathbf{F}_q^{\times}; b,c \in A^+ \\ d = \deg b > \deg c}} \frac{\alpha^{|J| - j}\delta^{\Sigma,j}_{J,s}\chi_{_I}(b)\chi_{_J}(c)}{b^ic^j}\\
         &= \sum \limits_{\substack{J \subseteq \Gamma ; I \sqcup J = \Sigma \sqcup \Gamma \\
i + j = s + t}}\delta^{\Gamma,j}_{J,t} \left(\sum \limits_{\alpha \in \mathbf{F}_q^{\times}} \alpha^{|J| - j}\right)\left(\sum \limits_{\substack{a,c \in A^+ \\ d = \deg a > \deg c}} \frac{\chi_{_I}(a)\chi_{_J}(c)}{a^ic^j}\right) \\
&+ \sum \limits_{\substack{J \subseteq \Sigma ; I \sqcup J = \Sigma \sqcup \Gamma \\
i + j = s + t}}\delta^{\Sigma,j}_{J,s}\left(\sum \limits_{\alpha \in \mathbf{F}_q^{\times}} \alpha^{|J| - j}\right) \left(\sum \limits_{\substack{b,c \in A^+ \\ d = \deg b > \deg c}} \frac{\chi_{_I}(b)\chi_{_J}(c)}{b^ic^j} \right)\\
        &=  \sum \limits_{\substack{J \subseteq \Gamma ; I \sqcup J = \Sigma \sqcup \Gamma \\
i + j = s + t}} \Delta^{\Gamma,j}_{J,t} S_d \begin{pmatrix}
 I  &   J\\
i & j  \end{pmatrix} + \sum \limits_{\substack{J \subseteq \Sigma ; I \sqcup J = \Sigma \sqcup \Gamma\\
i + j = s + t}} \Delta^{\Sigma,j}_{J,s} S_d \begin{pmatrix}
 I  &   J\\
i & j  \end{pmatrix}.
    \end{align*}
    The last equality follows from the following identity:
    \begin{align*}
\sum \limits_{\alpha \in \mathbf{F}_q^\times} \alpha^n =\begin{cases} -1, & \quad \text{if } n \equiv 0 \ (\text{mod } q - 1), \\
0, & \quad \text{otherwise},
\end{cases}
\end{align*}
for $n \in \mathbf{N}$. This proves the proposition.
\end{proof}

\begin{remark} 
When $ s = t = 1$, we recover the sum-shuffle formula for power sums of depth $1$ due to Pellarin (see \cite[Theorem 3.1]{Pel17}). When $ \Sigma = \Gamma =\emptyset$, we recover the sum-shuffle formula for power sums of depth $1$ due to Chen (see \cite[Remark 3.2]{Che15}).
\end{remark}

Let $d$ tend to infinity, we obtain the following explicit sum-shuffle formula for the product of two zeta values.

\begin{theorem} \label{thm: sum-shuffle formula}
Let $\Sigma,\Gamma$ be two finite weighted subsets of $\mathbf{N}^*$, and let $s,t$ be two positive integers. We have  
\begin{align*}
    \zeta_A \begin{pmatrix}
 \Sigma \\
s  \end{pmatrix}  \zeta_A \begin{pmatrix}
 \Gamma \\
t  \end{pmatrix} = & \ \zeta_A \begin{pmatrix}
 \Sigma  &   \Gamma\\
s & t  \end{pmatrix} + \zeta_A \begin{pmatrix}
 \Gamma  &   \Sigma\\
t & s  \end{pmatrix} + \zeta_A \begin{pmatrix}
 \Sigma \sqcup \Gamma \\
s + t  \end{pmatrix} \\
& + \sum \limits_{\substack{J \subseteq \Gamma ; I \sqcup J = \Sigma \sqcup \Gamma \\
i + j = s + t}} \Delta^{\Gamma,j}_{J,t} \zeta_A \begin{pmatrix}
 I  &   J\\
i & j  \end{pmatrix} + \sum \limits_{\substack{J \subseteq \Sigma ; I \sqcup J = \Sigma \sqcup \Gamma\\
i + j = s + t}} \Delta^{\Sigma,j}_{J,s} \zeta_A \begin{pmatrix}
 I  &   J\\
i & j  \end{pmatrix},
\end{align*}
where the indices $i,j$ run through all positive integers.
\end{theorem}

\begin{proposition} \label{prop: shuffle relations MZV power sum}
Let $\mathtt{A}, \mathtt{B}$ be two admissible arrays which are compositions of $\begin{pmatrix}
 \Sigma_1  \\
w_1  \end{pmatrix}, \begin{pmatrix}
 \Sigma_2  \\
w_2  \end{pmatrix}$, respectively. The following assertions hold: 
\begin{enumerate}[label=\textnormal{(\roman*)}, align= left]
\item There exist elements $\alpha_i \in \mathbf{F}_q$ and admissible arrays $\mathtt{C}_i$, which are compositions of $\begin{pmatrix}
 \Sigma_1 \sqcup \Sigma_2  \\
w_1 + w_2  \end{pmatrix}$, such that for all $d \in \mathbf{N}$,
\begin{equation*}
        S_d(\mathtt{A}) S_d(\mathtt{B}) = \sum \limits_i \alpha_i S_d(\mathtt{C}_i).
    \end{equation*}
    \item There exist elements $\alpha'_i \in \mathbf{F}_q$ and admissible arrays $\mathtt{C}'_i$, which are compositions of $\begin{pmatrix}
 \Sigma_1 \sqcup \Sigma_2  \\
w_1 + w_2  \end{pmatrix}$, such that for all $d \in \mathbf{N}$,
\begin{equation*}
        S_{<d}(\mathtt{A}) S_{<d}(\mathtt{B}) = \sum \limits_i \alpha'_i S_{<d}(\mathtt{C}'_i).
    \end{equation*}
    \item There exist elements $\alpha''_i \in \mathbf{F}_q$ and admissible arrays $\mathtt{C}''_i$, which are compositions of $\begin{pmatrix}
 \Sigma_1 \sqcup \Sigma_2  \\
w_1 + w_2  \end{pmatrix}$, such that for all $d \in \mathbf{N}$,
\begin{equation*}
        S_d(\mathtt{A}) S_{<d}(\mathtt{B}) = \sum \limits_i \alpha''_i S_d(\mathtt{C}''_i).
    \end{equation*}
\end{enumerate}
Here the elements $\alpha_i, \alpha'_i, \alpha''_i$ are independent of the degree $d$.
\end{proposition}

\begin{proof}
    Using Proposition \ref{prop: sum-shuffle formula}, the proof of Part (i) and Part (ii) is proceeded by induction on $w_1 + w_2$ and follows from similar arguments as in the proof of \cite[Proposition 2.1]{ND21}. Part (iii) follows from Part (ii) and \eqref{eq: power sum formula 3}.
\end{proof}

\section{A formula for power sums of Pellarin's multiple zeta values} \label{sec: a formula for power sums}

Throughout this section, we restrict our attention to the case $\Sigma$ is a usual finite subset of $\mathbf{N}^*$ with $|\Sigma| < q$ (see Remark \ref{rmk: finite subsets as fws}). In \cite{GP21}, Gezmis and Pellarin derived a formula \cite[Formula (22)]{GP21} for the power sums of Pellarin's MZVs using partial higher divided derivatives, but their formula was found to be incorrect, which partially affected the proof of \cite[Corollary 5.4]{GP21}. In this section, we provide a corrected version of their formula using a different approach (see Proposition \ref{prop: explicit formula}) and then use this formula to reprove  \cite[Corollary 5.4]{GP21} (see Proposition \ref{prop: Z = Z+}).

%%%%%%%%%%%%%%%%%%%%%%

\subsection{Preliminary results}
${}$
\par Set $\ell_0 = 1$ and $\ell_d = \prod^d_{i=1}(\theta - \theta^{q^i})$ for all $d \in \mathbf{N}^*$.  Set $D_0 = 1$ and $D_d = \prod^{d-1}_{i=0}(\theta^{q^d} - \theta^{q^i})$ for all $d \in \mathbf{N}^*$. We set $E_0(x) = x$, and for $d \in \mathbf{N}^*$,
\begin{equation*}
    E_d(x) = D_d^{-1}\prod \limits_{a \in A_<(d)} (x + a).
\end{equation*}
 Based on a result due to Carlitz (see \cite[Theorem 2.1]{Car35}), one deduces the following expansion:
\begin{equation}\label{eq: ed-formula}
    E_d(x) = \sum \limits_{k = 0}^d \dfrac{x^{q^k}}{D_k \ell_{d-k}^{q^k}}.
\end{equation}
In the following proposition, we recall some properties of the polynomial $E_d(x)$. For the proof and further properties of this polynomial, we refer the reader to \cite[Section 3.5]{Gos96}.
\begin{proposition} \label{prop: Gos}
For all $d \in \mathbf{N}$, we have the following properties:
\begin{enumerate}[label=\textnormal{(\roman*)}, align= left]
    \item $E_d(x)$ is an $\mathbf{F}_q$-linear polynomial.
    \item For $a \in A$ such that $\deg a < d$, $E_d(a) = 0$.
    \item $E_d(\theta^d) = 1$. 
\end{enumerate}
\end{proposition}
\noindent From Carlitz's work (see \cite[Section 9]{Car35}), we obtain the following generating function:
\begin{equation} \label{eq: Carlitz generating function}
    \dfrac{1}{l_d(1 - E_d(x))} = \sum\limits_{n \geq 0} S_d(n + 1) x^n.
\end{equation}

Set $b_0(t) = 1 $ and $ b_d(t) = \prod _{i = 0}^{d-1}(t-\theta^{q^i})$ for all $d \in \mathbf{N}^*$. For $d \in \mathbf{N}^*$, we define
\begin{equation*}
    P_d(t,x) = \sum \limits_{j = 0}^{d-1}b_j(t)E_j(x).
\end{equation*}
Note that the polynomial $P_d(t,x)$ is $\mathbf{F}_q$-linear in indeterminate $x$ by Proposition \ref{prop: Gos}(i). When $x = \theta^d$, one deduces from Proposition \ref{prop: Gos}(iii) and the same argument as in the proof of \cite[Corollary 2.11]{Per14} that 
\begin{equation}  \label{eq: values of  P_d(t,x)}
        P_d(t,\theta^d) = \sum \limits_{j = 0}^{d-1}b_j(t)E_j(\theta^d) = \chi_t(\theta^d) - b_d(t).
\end{equation}
The following result is deduced from a formula due to Perkins (see \cite[Proposition 2.17]{Per14}).
\begin{proposition} \label{prop: Perkins's formula}
Let $J$ be a finite subset of $\mathbf{N}^*$ such that $|J| < q$. For all $d \in \mathbf{N}^*$, we have
\begin{equation*}
    \sum \limits_{a \in A_<(d)} \frac{\chi_{_J}(a) \ell_d E_d(x - a)}{x - a} = \prod \limits_{j \in J} P_d(t_j,x).
\end{equation*}
\end{proposition}

\begin{proof}
     See \cite[Corollary 2.13, Proposition 2.17]{Per14}.
\end{proof}
 \noindent Using \eqref{eq: ed-formula}, we may write the polynomial $P_d(t,x)$ as a polynomial in indeterminate $x$ as follows:
\begin{equation*}
    P_d(t,x) = \sum \limits_{k = 0}^{d-1}\sum \limits_{j = k}^{d-1}\dfrac{b_j(t)}{D_k\ell_{j-k}^{q^k}}x^{q^k}.
\end{equation*}
For the convenience of writing, we denote the coefficient $\sum \limits_{j = k}^{d-1}\dfrac{b_j(t)}{D_k\ell_{j-k}^{q^k}}$ of $x^{q^k}$ by $Q_{d,k}(t)$. When $k = 0$, one deduces from \cite[Lemma 8]{Pel16} that 
\begin{equation} \label{eq: values of  Q_{d,0}(t)}
   Q_{d,0}(t) = \sum \limits_{j = 0}^{d-1} \frac{b_j(t)}{\ell_j} = \frac{b_d(t)}{\ell_{d-1}(t-\theta)}.
\end{equation}

%%%%%%%%%%%%%%%%%%%%%%%%%%%%%%

\subsection{Main result}
${}$
\par Let $\Sigma$ be a finite subset of $\mathbf{N}^*$. For the convenience of writing, we define
\begin{equation*}
    b_d(\Sigma) =  \prod \limits_{i \in \Sigma}b_d(t_i).
\end{equation*}

\begin{proposition} \label{prop: explicit formula}
Let $\Sigma$ be a finite subset of $\mathbf{N}^*$ such that $|\Sigma| < q$, and let $n$ be a natural number. For all $d \in \mathbf{N}$, we have 
 \begin{align*}
     S_d \begin{pmatrix}
     \Sigma \\ n+1
     \end{pmatrix} &= S_d(n+1)b_d(\Sigma)  + \sum \limits_{\substack{J \subsetneq \Sigma \\ I \sqcup J =  \Sigma}}  \sum \limits_{\substack{0 \leq k_i \leq d - 1, \ i \in I \\ n -  {\textstyle\sum_{i\in I}} q^{k_i} + 1 > 0}} S_d(n -  \textstyle\sum_{i\in I} q^{k_i} + 1) \prod \limits_{i \in I} Q_{d,k_i}(t_i) b_d(J).
 \end{align*}
 \end{proposition}

 \begin{proof}
    The case $d = 0$ is trivial. For the case $d > 0$, we claim that 
     \begin{equation} \label{eq: explicit formula 1}
         \sum \limits_{a \in A^+(d)} \dfrac{\chi_{_\Sigma}(a)}{a-x} = \sum \limits_{n\geq 0} S_d\begin{pmatrix}
    \Sigma \\ n+1
    \end{pmatrix} x^n.
     \end{equation}
     Indeed, expanding the geometric series, we have 
     \begin{align*}
    \sum \limits_{a \in A^+(d)} \dfrac{\chi_{_\Sigma}(a)}{a-x}  =  \sum \limits_{a \in A^+(d)} \dfrac{\chi_{_\Sigma}(a)}{a} \sum \limits_{n \geq 0} \left(\dfrac{x}{a}\right)^n = \sum \limits_{n \geq 0} \left( \sum \limits_{a \in A^+(d)} \dfrac{\chi_{_\Sigma}(a)}{a^{n+1}} \right) x^n = \sum \limits_{n\geq 0} S_d\begin{pmatrix}
    \Sigma \\ n+1
    \end{pmatrix} x^n.
\end{align*}
On the other hand, it follows from Proposition \ref{prop: character}(i) that
\begin{align*}
    \sum \limits_{a \in A^+(d)} \dfrac{\chi_{_\Sigma}(a)}{a-x}  =   \sum \limits_{a \in A_<(d)} \dfrac{\chi_{_\Sigma}(\theta^d + a)}{\theta^d + a - x}
    &= \sum \limits_{a \in A_<(d)} \dfrac{\sum\limits_{I \sqcup J = \Sigma}\chi_{_I}(\theta^d)\chi_{_J}(a)}{\theta^d + a - x} \\
    &= \sum\limits_{I \sqcup J = \Sigma} \chi_{_I}(\theta^d) \sum \limits_{a \in A_<(d)}  \dfrac{\chi_{_J}(a)}{\theta^d + a - x}.
\end{align*}
Replacing $x$ by $x - \theta^d$ in Proposition \ref{prop: Perkins's formula} and using Proposition \ref{prop: Gos}, one deduces that
\begin{equation*}
     \sum \limits_{a \in A_<(d)}  \dfrac{\chi_{_J}(a)}{\theta^d + a - x} = \dfrac{1}{\ell_d (1 - E_d(x))} \prod \limits_{j \in J} P_d(t_j,x - \theta^d).
\end{equation*}
Thus we have
\begin{align} \label{eq: explicit formula 2}
    \sum \limits_{a \in A^+(d)} \dfrac{\chi_{_\Sigma}(a)}{a-x}  &= \dfrac{1}{\ell_d (1 - E_d(x))} \sum\limits_{I \sqcup J = \Sigma} \chi_{_I}(\theta^d) \prod \limits_{j \in J} P_d(t_j,x - \theta^d)\\ \notag
    &= \dfrac{1}{\ell_d (1 - E_d(x))}  \prod \limits_{i \in \Sigma} (\chi_{t_i}(\theta^d) + P_d(t_i,x - \theta^d)) \\ \notag
    &= \dfrac{1}{\ell_d (1 - E_d(x))}  \prod \limits_{i \in \Sigma} (P_d(t_i,x) + b_d(t_i)).
\end{align}
The last equality follows from \eqref{eq: values of  P_d(t,x)}. Combining \eqref{eq: explicit formula 1}, \eqref{eq: explicit formula 2}, and \eqref{eq: Carlitz generating function}, we obtain
\begin{equation*} 
    \sum \limits_{n\geq 0} S_d\begin{pmatrix}
    \Sigma \\ n+1
    \end{pmatrix} x^n = \left( \sum \limits_{n\geq 0} S_d(n+1) x^n \right) \prod_{i \in \Sigma}\left(\sum \limits_{k = 0}^{d-1}Q_{d,k}(t_i)x^{q^k} + b_d(t_i) \right).
\end{equation*}
The result then follows from equating the coefficient of $x^n$ on both sides of the above identity.
 \end{proof}

%%%%%%%%%%%%%%%%%%%%%%%%%%%%%%

\subsection{Verification and comparison} \label{subsec: verification and comparison}
${}$
\par In this subsection, we verify and compare two formulas for power sums of Pellarin's MZVs: one originally proposed by Gezmis and Pellarin in \cite[Equation (22)]{GP21}, and our formula given in Proposition \ref{prop: explicit formula}. We will show that Proposition \ref{prop: explicit formula} provides the correct result. 

For the convenience of the reader, we recall the formula of Gezmis and Pellarin as follows. Let $\Sigma$ be a finite subset of $\mathbf{N}^*$ such that $|\Sigma| < q$, and let $n$ be a natural number. For all $d \in \mathbf{N}$, we have
\begin{align*} 
     S_d \begin{pmatrix}
     \Sigma \\ n+1
     \end{pmatrix} &= S_d(n+1)b_d(\Sigma)  + \sum \limits_{I \subsetneq \Sigma } (-1)^{|I|}  \sum \limits_{\substack{k_i \in \mathbf{N}, \ i \in I \\ n -  {\textstyle\sum_{i\in I}} q^{k_i} + 1 > 0}} S_d(n -  \textstyle\sum_{i\in I} q^{k_i} + 1) \prod \limits_{i \in I} Q_{d,k_i}(t_i).
 \end{align*}
We note that the polynomial $Q_{d,k}(t)$ used in our work is denoted by $H^{(d)}_{k}(t)$ in \cite{GP21}.

We consider the case $\Sigma = \{1,2\}, n = 1$ and $q = 3$, so that $|\Sigma| = 2 < 3$. From the formula of Gezmis and Pellarin, we have
\begin{align} \label{eq: result 1}
    S_d\begin{pmatrix}
\{1,2\} \\ 2
\end{pmatrix} &=  S_d(2)b_d(t_1)b_d(t_2) - S_d(1)Q_{d,0}(t_1) - S_d(1)Q_{d,0}(t_2) \\
&= \frac{b_d(t_1)b_d(t_2)}{\ell_d^2} - \frac{b_d(t_1)}{\ell_d \ell_{d-1}(t_1 - \theta)} - \frac{b_d(t_2)}{\ell_d \ell_{d-1}(t_2 - \theta)}. \notag
\end{align}
The last equality follows from \eqref{eq: values of  Q_{d,0}(t)} and the fact that $S_d(s) = 1/l_d^s$ for all positive integers $s \leq q$ (see \cite[Section 3.3]{Tha09}). From Proposition \ref{prop: explicit formula} and the above arguments, we have
\begin{align} \label{eq: result 2}
    S_d\begin{pmatrix}
\{1,2\} \\ 2
\end{pmatrix} &=  S_d(2)b_d(t_1)b_d(t_2) + S_d(1)Q_{d,0}(t_1)b_d(t_2) + S_d(1)Q_{d,0}(t_2)b_d(t_1) \\
&= \frac{b_d(t_1)b_d(t_2)}{\ell_d^2} + \frac{b_d(t_1)b_d(t_2)}{\ell_d \ell_{d-1}(t_1 - \theta)} + \frac{b_d(t_1)b_d(t_2)}{\ell_d \ell_{d-1}(t_2 - \theta)}. \notag
\end{align}

For the verification of the results \eqref{eq: result 1} and \eqref{eq: result 2}, we now give an alternative method to compute $S_d\begin{pmatrix}
\{1,2\} \\ 2
\end{pmatrix}$. From Proposition \ref{prop: sum-shuffle formula} (see also \cite[Theorem 3.1]{Pel17}), we have
\begin{equation*} 
     S_d\begin{pmatrix}
\{1,2\} \\ 2
\end{pmatrix} = S_d\begin{pmatrix}
\{1\} \\ 1
\end{pmatrix}S_d\begin{pmatrix}
\{2\} \\ 1
\end{pmatrix} + S_d\begin{pmatrix}
\{2\} & \{1\} \\ 1 & 1
\end{pmatrix} + S_d\begin{pmatrix}
\{1\} & \{2\} \\ 1 & 1
\end{pmatrix}.
\end{equation*}
Based on a result due to Pellarin (see \cite[Lemma 5.1]{Pel17}), one deduces that $S_d\begin{pmatrix}
\{1\} \\ 1
\end{pmatrix} = \dfrac{b_d(t_1)}{l_d}$ and $S_d\begin{pmatrix}
\{2\} \\ 1
\end{pmatrix} = \dfrac{b_d(t_2)}{l_d}$.
Moreover, it follows from \eqref{eq: power sum formula 3} and \cite[Lemma 8]{Pel16} that
\begin{align*}
    S_d\begin{pmatrix}
\{2\} & \{1\} \\ 1 & 1
\end{pmatrix} &= S_{d}\begin{pmatrix}
\{2\} \\ 1
\end{pmatrix}S_{<d}\begin{pmatrix}
\{1\} \\ 1
\end{pmatrix} = \frac{b_d(t_2)b_d(t_1)}{\ell_{d}\ell_{d-1}(t_1-\theta)},\\
S_d\begin{pmatrix}
\{1\} & \{2\} \\ 1 & 1
\end{pmatrix} &= S_{d}\begin{pmatrix}
\{1\} \\ 1
\end{pmatrix}S_{<d}\begin{pmatrix}
\{2\} \\ 1
\end{pmatrix} = \frac{b_d(t_1)b_d(t_2)}{\ell_{d}\ell_{d-1}(t_2-\theta)}.
\end{align*}
This proves that
\begin{align*}
      S_d\begin{pmatrix}
\{1,2\} \\ 2
\end{pmatrix} = \frac{b_d(t_1)b_d(t_2)}{\ell_d^2} + \frac{b_d(t_1)b_d(t_2)}{\ell_d \ell_{d-1}(t_1 - \theta)} + \frac{b_d(t_1)b_d(t_2)}{\ell_d \ell_{d-1}(t_2 - \theta)},
\end{align*}
which differs from \eqref{eq: result 1} and leads to the same result as in \eqref{eq: result 2}.

\subsection{Dagger multiple zeta values} 
${}$
\par In this subsection, we will reprove a result due to Gezmis and Pellarin \cite[Corollary 5.4]{GP21} using Proposition \ref{prop: explicit formula}. To do so, we first review some notions introduced by Gezmis and Pellarin.

% Let $\Sigma$ be a finite subset of $\mathbf{N}^*$ such that $|\Sigma| < q$, and let $s$ be a positive integer. For $d \in \mathbf{N}$, we define
% \begin{align*}
%     S_d^{\dagger}\begin{pmatrix}
% \Sigma\\
% s
% \end{pmatrix}  = S_d(s)b_d(\Sigma).
% \end{align*}
Let $\Sigma$ be a finite subset of $\mathbf{N}^*$ such that $|\Sigma| < q$. We denote by $\mathbf{E}_{\Sigma}$ the $\mathbf{C}_{\infty}$-subalgebra of $\mathbf{T}_{\Sigma}(\mathbf{C}_{\infty})$ consisting of all entire functions in variables $t_i$ with $i \in \Sigma$. For any admissible array $\mathtt{A} =  \begin{pmatrix}
\Sigma_1 & \dotsb & \Sigma_r\\
s_1 & \dotsb & s_r
\end{pmatrix}$ of type $\Sigma$, we define
\begin{align*}
 \zeta_A^{\dagger}(\mathtt{A}) = \zeta_A^{\dagger}\begin{pmatrix}
\Sigma_1 & \dotsb & \Sigma_r\\
s_1 & \dotsb & s_r
\end{pmatrix} = \sum\limits_{d_1 > \cdots > d_r \geq 0} S_{d_1}(s_1)b_{d_1}(\Sigma_1) \cdots S_{d_r}(s_r)b_{d_r}(\Sigma_r).
\end{align*}
One verifies at once that the above series converges to an entire function in $\mathbf{E}_{\Sigma}$. It should be remarked that this series does not converge when $|\Sigma| \geq  q$. For $d \in \mathbf{N}$, we define
\begin{align*}
    S_d^{\dagger}(\mathtt{A}) &= \sum\limits_{d = d_1 > \cdots > d_r \geq 0} S_{d_1}(s_1)b_{d_1}(\Sigma_1) \cdots S_{d_r}(s_r)b_{d_r}(\Sigma_r),\\
 S_{<d}^{\dagger}(\mathtt{A}) &= \sum\limits_{d > d_1 > \cdots > d_r \geq 0} S_{d_1}(s_1)b_{d_1}(\Sigma_1) \cdots S_{d_r}(s_r)b_{d_r}(\Sigma_r).
\end{align*}

Using similar arguments as for the power sums of Pellarin's MZVs, one may verify that $S_d^{\dagger}$ and $S_{<d}^{\dagger}$ satisfy the properties outlined as in Proposition \ref{prop: shuffle relations MZV power sum}. In particular, we obtain the following result:

\begin{proposition} \label{prop: shuffle relations DMZV power sum}
Let $\mathtt{A}, \mathtt{B}$ be two admissible arrays which are compositions of $\begin{pmatrix}
 \Sigma_1  \\
w_1  \end{pmatrix}, \begin{pmatrix}
 \Sigma_2  \\
w_2  \end{pmatrix}$, respectively, such that $|\Sigma_1 \sqcup \Sigma_2| < q$. There exist elements $\alpha_i \in \mathbf{F}_q$ and admissible arrays $\mathtt{C}_i$, which are compositions of $\begin{pmatrix}
 \Sigma_1 \sqcup \Sigma_2  \\
w_1 + w_2  \end{pmatrix}$, such that for all $d \in \mathbf{N}$,
\begin{equation*}
        S_d^{\dagger}(\mathtt{A}) S_{<d}^{\dagger}(\mathtt{B}) = \sum \limits_i \alpha_i S_d^{\dagger}(\mathtt{C}_i).
    \end{equation*}
    Here the elements $\alpha_i$ are independent of the degree $d$.
\end{proposition}

We denote by $\mathcal{Z}^{\dagger}_{w, \Sigma}$ the $\mathbf{F}_q$-vector subspace of $\mathbf{E}_{\Sigma}$ generated by $\zeta_A^{\dagger}(\mathtt{A})$, where $\mathtt{A}$ ranges over all admissible arrays of weight $w$ and type $\Sigma$. We let $\mathcal{Z}_{w, \Sigma}$ denote the $\mathbf{F}_q$-vector subspace of $\mathbf{E}_{\Sigma}$ generated by Pellarin's MZVs of weight $w$ and type $\Sigma$. We are now ready to reprove \cite[Corollary 5.4]{GP21} stated as follows:

\begin{proposition}  \label{prop: Z = Z+}
    Let $\Sigma$ be a finite subset of $\mathbf{N}^*$ such that $|\Sigma| < q$. For all positive integers $w$, we have $\mathcal{Z}_{w, \Sigma} = \mathcal{Z}^{\dagger}_{w, \Sigma}$.
\end{proposition}

The following lemma due to Gezmis and Pellarin will be useful. For the proof, we refer the reader to \cite[Lemma 5.5]{GP21}.

\begin{lemma} \label{lem: Q = S<d}
Let $I$ be a finite subset of $\mathbf{N}^*$ such that $|I| < q$. For all $ d\in \mathbf{N}$, we can write 
\begin{equation*}
    \prod \limits_{i \in I} Q_{d,k_i}(t_i) = \sum\limits_j \alpha_j S_{<d}(\mathtt{C}_j),
\end{equation*}
where $\alpha_j$ are elements in $\mathbf{F}_q$, which are independent of the degree $d$, and $\mathtt{C}_j$ are admissible arrays, which are compositions of $\begin{pmatrix}
I\\
\textstyle\sum_{i\in I} q^{k_i}
\end{pmatrix}$.
\end{lemma}

\begin{proposition} \label{prop: Z=Z+}
Let $\Sigma$ be a finite subset of $\mathbf{N}^*$ such that $|\Sigma| < q$, and let $w$ be a positive integer. Let $\mathtt{A}$ be an admissible array which is a composition of $\begin{pmatrix}
\Sigma\\
w
\end{pmatrix}$. The following statements hold:
\begin{enumerate} [label=\textnormal{(\roman*)}, align= left]
\item There exist elements $\alpha_i \in \mathbf{F}_q$ and admissible arrays $\mathtt{B}_i$, which are compositions of $\begin{pmatrix}
\Sigma\\
w
\end{pmatrix}$, such that for all $ d\in \mathbf{N}$,
    \begin{equation*}
        S_d(\mathtt{A}) = \sum \limits_i \alpha_i S_d^{\dagger}(\mathtt{B}_i).
    \end{equation*}

\item There exist elements $\alpha'_i \in \mathbf{F}_q$ and admissible arrays $\mathtt{B}'_i$, which are compositions of $\begin{pmatrix}
\Sigma\\
w
\end{pmatrix}$, such that for all $ d\in \mathbf{N}$,
    \begin{equation*}
        S_d^{\dagger}(\mathtt{A}) = \sum \limits_i \alpha'_i S_d(\mathtt{B}'_i).
    \end{equation*}
\end{enumerate}
Here the elements $\alpha_i$ and $\alpha'_i$ are independent of the degree $d$.
\end{proposition}

\begin{proof}
    We proceed the proof by induction on $w$. For the case $w = 1$, i.e., $\mathtt{A} = \begin{pmatrix}
\Sigma\\
1
\end{pmatrix}$, it follows from Proposition \ref{prop: explicit formula} that
\begin{equation*}
    S_d\begin{pmatrix}
 \Sigma \\
1  \end{pmatrix} = S_d^{\dagger}\begin{pmatrix}
 \Sigma \\
1  \end{pmatrix},
\end{equation*}
which proves the base step. Assume that Proposition \ref{prop: Z=Z+} holds for $w < n$ with $n \in \mathbf{N}$ and $n \geq 2$. We need to show that Proposition \ref{prop: Z=Z+} holds for $w = n$. We consider two cases:

\textit{Case $1$}:  The admissible array $\mathtt{A}$ has depth $1$, i.e., $\mathtt{A} = \begin{pmatrix}
\Sigma\\
n
\end{pmatrix}$. Proposition \ref{prop: explicit formula} shows that
\begin{equation}  \label{eq: Z=Z+}
     S_d \begin{pmatrix}
     \Sigma \\ n
     \end{pmatrix} = S_d^{\dagger} \begin{pmatrix}
     \Sigma \\ n
     \end{pmatrix}  + \sum \limits_{\substack{J \subsetneq \Sigma \\ I \sqcup J =  \Sigma}}  \sum \limits_{\substack{0 \leq k_i \leq d - 1, \ i \in I \\ n -  {\textstyle\sum_{i\in I}} q^{k_i}> 0}} S_d^{\dagger} \begin{pmatrix}
     J \\ n -  \textstyle\sum_{i\in I} q^{k_i}
     \end{pmatrix}\prod \limits_{i \in I} Q_{d,k_i}(t_i).
 \end{equation}
 Since ${\textstyle\sum_{i\in I}} q^{k_i} < n$, the induction hypothesis and Lemma \ref{lem: Q = S<d} shows that there exist elements $\gamma_j \in \mathbf{F}_q$ and admissible arrays $\mathtt{C}_j$, which are compositions of $\begin{pmatrix}
I\\
\textstyle\sum_{i\in I} q^{k_i}
\end{pmatrix}$, such that for all $ d\in \mathbf{N}$, $\prod_{i\in I} Q_{d,k_i}(t_i) = \sum_j \gamma_j S_{<d}^{\dagger}(\mathtt{C}_j)$. Thus Part (i) follows from \eqref{eq: Z=Z+} and Proposition \ref{prop: shuffle relations DMZV power sum}. Since $n - {\textstyle\sum_{i\in I}} q^{k_i} < n$, the induction hypothesis shows that there exist elements $\gamma'_i \in \mathbf{F}_q$ and admissible arrays $\mathtt{C}'_i$, which are compositions of $\begin{pmatrix}
\Sigma\\
w
\end{pmatrix}$, such that for all $ d\in \mathbf{N}$, $S_d^{\dagger} \begin{pmatrix}
     J \\ n -  \textstyle\sum_{i\in I} q^{k_i}
     \end{pmatrix} = \sum_i \gamma'_i S_d(\mathtt{C}'_i)$. Thus Part (ii) follows from \eqref{eq: Z=Z+}, Lemma \ref{lem: Q = S<d} and Proposition \ref{prop: shuffle relations MZV power sum}(iii).

\textit{Case $2$}:  The admissible array $\mathtt{A}$ has depth $> 1$. We may assume that $ \mathtt{A} = \begin{pmatrix}
\Sigma_1 & \dotsb & \Sigma_r\\
s_1 & \dotsb & s_r
\end{pmatrix}$, so that $s_1 + \cdots + s_r = n$. Since $s_1 < n$ and $s_2 + \cdots + s_r < n$, the induction hypothesis shows that there exist elements $\gamma_i, \gamma'_i \in \mathbf{F}_q$ and admissible arrays $\mathtt{C}_i$, which are compositions of $\begin{pmatrix}
\Sigma_1\\
s_1
\end{pmatrix}$, and admissible arrays $\mathtt{C}'_i$, which are compositions of $\begin{pmatrix}
\Sigma_2 \sqcup \dotsb \sqcup \Sigma_r\\
s_2 + \dotsb + s_r
\end{pmatrix}$, such that for all $ d\in \mathbf{N}$, $S_d\begin{pmatrix}
\Sigma_1\\
s_1
\end{pmatrix} = \sum \limits_i \gamma_i S_d^{\dagger}(\mathtt{C}_i)$ and $S_{<d}\begin{pmatrix}
\Sigma_2 & \dotsb & \Sigma_r\\
s_2 & \dotsb & s_r
\end{pmatrix} = \sum \limits_i \gamma'_i S_{<d}^{\dagger}(\mathtt{C}'_i)$. Thus Part (i) follows from \eqref{eq: power sum formula 3} and Proposition \ref{prop: shuffle relations DMZV power sum}. From similar arguments as above, one may verify that Part (ii) holds in this case.
\end{proof}

\begin{proof}[Proof of Proposition \ref{prop: Z = Z+}]
The result follows immediately from Proposition \ref{prop: Z=Z+} by letting $d$ tend to infinity.    
\end{proof}

%%%%%%%%%%%%%%%%%%%%%%

\section{Structure of $\ker(\mathcal{G}_{\Sigma})$} \label{sec: counterexamples}

% Let $\Sigma$ be a finite weighted subset of $\mathbf{N}^*$. Throughout this section, we restrict our attention to the case 
% $\Sigma$ is a usual finite subset of $\mathbf{N}^*$ with $|\Sigma| < q$ (see Remark \ref{rmk: finite subsets as fws}). We shall first review the notion of \textit{trivial multiple zeta values} introduced by Gezmis and Pellarin, and then give some counterexamples for \cite[Conjecture 6.15]{GP21} due to Gezmis and Pellarin.

\subsection{Some constructions}  \label{subsec: multiple polylogarithms}
${}$
\par 
Let $\Sigma$ be a fixed finite subset of $\mathbf{N}^*$ such that $|\Sigma| < q$. In this subsection, we briefly review the constructions of the maps $\mathcal{F}_{\Sigma}$ and $\mathcal{E}_{\Sigma}$ introduced by Gezmis and Pellarin in \cite{GP21}. 

Let $\{X_n\}_{n \in \mathbf{N}^*}$ be a sequence of indeterminates. For each finite subset $U$ of $\mathbf{N}^*$, we write $X_{U} = \prod_{i \in U} X_i$. For any admissible array $\mathtt{A} =  \begin{pmatrix}
\Sigma_1 & \dotsb & \Sigma_r\\
s_1 & \dotsb & s_r
\end{pmatrix}$ of type $\Sigma$, we define the \textit{multiple polylogarithm} associated with the array $\mathtt{A}$ as
\begin{align*}
 \lambda_A(\mathtt{A}) 
 % = \sum\limits_{d_1 > \cdots > d_r \geq 0} S_{d_1}(s_1) X_{\Sigma_1}^{q^{d_1}} \cdots S_{d_r}(s_r) X_{\Sigma_r}^{q^{d_r}}
 = \sum \limits_{\substack{a_i \in A^+\\  \deg a_1> \dotsb > \deg a_r \geq 0}}  \dfrac{X_{\Sigma_1}^{q^{\deg a_1}} \cdots X_{\Sigma_r}^{q^{\deg a_r}}}{a_1^{s_1} \cdots a_r^{s_r}} \in \mathbf{T}_{\Sigma}(\mathbf{C}_{\infty}).
\end{align*}

Recall that $\mathbf{E}_{\Sigma}$ is the $\mathbf{C}_{\infty}$-subalgebra of $\mathbf{T}_{\Sigma}(\mathbf{C}_{\infty})$ consisting of all entire functions in variables $t_i$ with $i \in \Sigma$. We let $\mathcal{Z}_{\Sigma}(K)$ (resp. $\mathcal{L}_{\Sigma}(K)$) denote the $K$-vector subspace of $\mathbf{E}_{\Sigma}$ generated by elements $\zeta_A(\mathtt{A})$ (resp. $\lambda_A(\mathtt{A})$), where $\mathtt{A}$ ranges over all admissible arrays of type $\Sigma$. When $\Sigma = \emptyset$, one verifies at once that $\mathcal{Z}_{\emptyset}(K) = \mathcal{L}_{\emptyset}(K)$, which is the $K$-algebra generated by Thakur's MZVs. Let $\mathcal{Z}_{n,\Sigma}(K)$ (resp. $\mathcal{L}_{n,\Sigma}(K)$) denote the $K$-vector subspace of $\mathbf{E}_{\Sigma}$ generated by elements $\zeta_A(\mathtt{A})$ (resp. $\lambda_A(\mathtt{A})$), where $\mathtt{A}$ ranges over all admissible arrays of weight $n$ and type $\Sigma$. Gezmis and Pellarin showed in \cite{GP21} that both $\mathcal{Z}_{\Sigma}(K)$ and $\mathcal{L}_{\Sigma}(K)$ are graded $\mathcal{Z}_{\emptyset}(K)$-module with grading given by the weight, so that 
 \begin{equation*}
   \mathcal{Z}_{\Sigma}(K) = \bigoplus \limits_{n = 0}^{\infty} \mathcal{Z}_{n,\Sigma}(K) \quad \text{and} \quad  \mathcal{L}_{\Sigma}(K) = \bigoplus \limits_{n = 0}^{\infty} \mathcal{L}_{n,\Sigma}(K).
 \end{equation*}
 We recall the construction of the morphisms $\mathcal{F}_{\Sigma},\mathcal{E}_{\Sigma}:\mathcal{Z}_{\Sigma}(K)\to\mathcal{L}_{\Sigma}(K)$ from \cite{GP21}.
\begin{theorem}[Theorem 5.2 \cite{GP21}]\label{zeta} Let $\Sigma$ be a finite subset of $\mathbf{N}^*$ such that $|\Sigma|<q$.
Let $f\in \mathcal{Z}_{\Sigma}(K)$ then $$\mathcal{F}_{\Sigma}(f):=\sum_{\underline{i}\in\mathbf{N}^{|\Sigma|}}\sum_{\underline{j}\leq\underline{i}\in\mathbf{N}^{|\Sigma|}}\frac{f(\theta^{q^{\underline{j}}})}{D_{\underline{j}}\ell_{\underline{i}-\underline{j}}^{q^{\underline{j}}}}X_{\Sigma}^{q^{\underline{i}}}\in \mathcal{L}_{\Sigma}(K),$$
where $\underline{i}-\underline{j}$ is the difference of $\underline{i}$ and $\underline{j}$ in the additive group $\mathbf{Z}^{|\Sigma|}$, $X_{\Sigma}^{q^{\underline{i}}}=\prod_{r\in\Sigma}X_j^{q^{i_r}}$, $\theta^{q^{\underline{j}}}=(\theta^{q^{j_r}})_{r\in\Sigma}$, $D_{\underline{j}}=\prod_{r\in\Sigma}D_{j_r}$ and $\ell_{\underline{j}}^{q^{\underline{i}}}=\prod_{r\in\Sigma}\ell_{j_r}^{q^{i_r}}$.
Moreover, $\mathcal{F}_{\Sigma}:\mathcal{Z}_{\Sigma}(K)\to\mathcal{L}_{\Sigma}(K)$ is an isomorphism of graded $\mathcal{Z}_{\emptyset}(K)$-modules.
\end{theorem}
\begin{theorem}[Corollary 6.9 \cite{GP21}]\label{lambda}Let $\Sigma$ be a finite subset of $\mathbf{N}^*$ such that $|\Sigma|<q$.
Let $f\in \mathcal{Z}_{\Sigma}(K)$ then $$\mathcal{E}_{\Sigma}(f)=\sum_{\underline{i}\in\mathbf{N}^{|\Sigma|}}\frac{f(\theta^{q^{\underline{i}}})}{D_{\underline{i}}}\prod_{j\in\Sigma}\lambda_{A}\begin{pmatrix}
\{j\} \\1
\end{pmatrix}^{q^{i_j}}\in \mathcal{L}_{\Sigma}(K)$$
and $\mathcal{E}_{\Sigma}=\mathcal{F}_{\Sigma}$.
\end{theorem}

 % We now restrict our attention to the case $\Sigma$ is a usual finite subset of $\mathbf{N}^*$ with $|\Sigma| < q$. We consider the map
 % \begin{equation*}
 %     \mathcal{F}_{\Sigma} \colon \mathcal{Z}_{\Sigma}(K) \rightarrow  \mathcal{L}_{\Sigma}(K),
 % \end{equation*}
 % which maps 
 
%%%%%%%%%%%%%%%%%%%%%%%%%%%%%%%%

\subsection{Trivial multiple zeta values} 
${}$
\par We continue with the same notation as in the preceding sections. We let $\mathcal{Z}^{\text{triv}}_{\Sigma,n}(K)$ denote the set of all elements $f \in \mathcal{Z}_{\Sigma,n}(K)$ satisfies $f(\theta^{q^{\underline{k}}}) = 0$ for all but finitely many tuples $\underline{k} \in \mathbf{N}^{|\Sigma|}$. One verifies easily that $\mathcal{Z}^{\text{triv}}_{\Sigma,n}(K)$ is a $K$-vector subspace of $ \mathcal{Z}_{\Sigma,n}(K)$. Moreover, we have
\begin{equation*}
    \mathcal{Z}^{\text{triv}}_{\Sigma}(K) = \bigoplus_{n = 0}^{\infty}\mathcal{Z}^{\text{triv}}_{\Sigma,n}(K)
\end{equation*}
is a $ \mathcal{Z}_{\emptyset}(K)$-submodule of $ \mathcal{Z}_{\Sigma}(K)$. 
\begin{definition}Let $\text{ev}:\mathcal{L}_{\Sigma}(K)\to \mathcal{Z}_{\emptyset}(K)$ be the evaluation map sending $X_i$ to $1$ for all $i\in\Sigma$. Then, we denote by $\mathcal{G}_{\Sigma}$  the map $\text{ev}\circ\mathcal{F}_{\Sigma}|_{\mathcal{Z}^{\text{triv}}_{\Sigma}(K)}$.
\end{definition}
 
 We recall the following result due to Gezmis and Pellarin. 
\begin{lemma} \label{lem: special zeta}
For all natural numbers $k$, we have
\begin{equation*}
    \frac{(-1)^k b_k(t)\zeta_A \begin{pmatrix}
    t\\ q^k
    \end{pmatrix}}{D_k} = \zeta_A  \begin{pmatrix}
    t & \emptyset &  \cdots & \emptyset\\ 1 & q-1 &  \cdots &  (q-1)q^{k-1} 
    \end{pmatrix}.
\end{equation*}
\end{lemma}

\begin{proof}
    See \cite[Lemma 6.12]{GP21}.
\end{proof}

\begin{lemma} \label{lem: eta}
For all natural numbers $k$, we set
\begin{equation*}
    \eta_k(t) = \frac{b_k(t)\zeta_A \begin{pmatrix}
    t\\ q^k
    \end{pmatrix}}{D_k}.
\end{equation*}
Then 
\begin{align*}
    \eta_k(\theta^{q^i}) = \begin{cases}  1 & \quad \text{if } i = k, \\
0 & \quad \text{if } i \ne k.
\end{cases}
\end{align*}
\end{lemma}
\begin{proof}
We first note that 
\begin{equation*}
    \zeta_A \begin{pmatrix}
    t\\ q^k
    \end{pmatrix} \Biggr|_{t  = \theta^{q^i}} = \sum \limits_{a \in A^{+}} \frac{a(\theta)^{q^i}}{a^{q^k}} = \sum \limits_{a \in A^{+}} \frac{1}{a^{q^k - q^i}} = \zeta_A( q^k - q^i).
\end{equation*}
If $i = k$, then $b_k(\theta^{q^k}) = D_k $ and $\zeta_A \begin{pmatrix}
    t\\ q^k
    \end{pmatrix}\Biggr|_{t  = \theta^{q^k}} = \zeta_A( 0) = 1$, hence $\eta_k(\theta^{q^k}) = 1$. If $i <k$, then $b_k(\theta^{q^i}) = 0$, hence $\eta_k(\theta^{q^i}) = 0$. If $i > k$, then $q^k - q^i$ is a negative integer satisfies $(q - 1) | (q^k - q^i)$, hence $\zeta_A \begin{pmatrix}
    t\\ q^k
    \end{pmatrix} \Biggr|_{t  = \theta^{q^i}} =  \zeta_A( q^k - q^i) = 0 $ (See \cite[Theorem 5.3]{Gos79}), showing that $\eta_k(\theta^{q^i}) = 0$. This proves the lemma.
\end{proof}
\begin{corollary} \label{generators} 
For each tuples $\underline{k} \in \mathbf{N}^{|\Sigma|}$, we set
\begin{equation*}
    \eta_{\underline{k}} = \prod\limits_{j \in \Sigma} \eta_{k_j}(t_j).
\end{equation*}
Then $ \eta_{\underline{k}} $ is an element in $\mathcal{Z}^{\rm triv \rm}_{\Sigma,\sum_{j \in \Sigma}q^{k_j}}(K)$. Moreover, $\mathcal{Z}^{\rm triv \rm}_{\Sigma}(K)$ is equal to its $\mathcal{Z}_{\emptyset}(K)$- submodule generated by $\{\eta_{\underline{k}}| \underline{k}\in  \mathbf{N}^{|\Sigma|}\}$.
\end{corollary}
\begin{proof} \label{rmk: eta tuple}

It follows from Lemma \ref{lem: special zeta} that for each $j \in \Sigma$,
\begin{equation*}
    \eta_{k_j}(t_j) =(-1)^{k_j}\zeta_A  \begin{pmatrix}
    t_j & \emptyset &  \cdots & \emptyset\\ 1 & q-1 &  \cdots &  (q-1)q^{k_j-1} 
    \end{pmatrix} 
\end{equation*}
Note that $ 1 + (q-1) + \cdots + (q-1)q^{k_j-1} = q^{k_j} $. Thus from Theorem  \ref{thm: sum-shuffle formula}, we have
\begin{equation*}
    \eta_{\underline{k}} = (-1)^{\sum_{j \in \Sigma}k_j} \prod\limits_{j \in \Sigma} \zeta_A  \begin{pmatrix}
    t_j & \emptyset &  \cdots & \emptyset\\ 1 & q-1 &  \cdots &  (q-1)q^{k_j-1} 
    \end{pmatrix} \in \mathcal{Z}_{\Sigma,\sum_{j \in \Sigma}q^{k_j}}(K).
\end{equation*}
 Moreover, one deduces from Lemma \ref{lem: eta} that
\begin{align*}
    \eta_{\underline{k}}(\theta^{q^{\underline{i}}}) = \begin{cases}  1 & \quad \text{if } \underline{i} = \underline{k}, \\
0 & \quad \text{otherwise}.
\end{cases}
\end{align*}
This shows that $ \eta_{\underline{k}} $ is an element in $\mathcal{Z}^{\text{triv}}_{\Sigma,\sum_{j \in \Sigma}q^{k_j}}(K)$. The second part of this corollary was shown in \cite[Theorem 6.10]{GP21}.
\end{proof}

%%%%%%%%%%%%%%%

\subsection{The structure of $\ker(\mathcal{G}_{\Sigma})$} \label{subsec: some counterexamples}
${}$
\par In this subsection, we study the structure of $\mathcal{G}_{\Sigma}$ for a finite subset $\Sigma$ of $\mathbf{N}^*$ with $|\Sigma|<q$. %Note that one deduces from \cite[Corollary 6.9]{GP21} another formula for $\mathcal{F}_{\Sigma}$ as follow: for each element $f \in \mathcal{Z}_{\Sigma}(K)$, we have
%\begin{equation} \label{eq: another formula of f sigma}
%    \mathcal{F}_{\Sigma}(f) = \sum\limits_{\underline{i}\in \mathbf{N}^{|\Sigma|}} \frac{f(\theta^{q^{\underline{i}}})}{D_{\underline{i}}} \prod \limits_{j \in \Sigma} \alpha_A\begin{pmatrix}
   % X_j \\ 1
   % \end{pmatrix}^{q^{i_j}}.
%\end{equation}

Note that $\text{ev}(\lambda_{A}\begin{pmatrix}
\{j\} \\1
\end{pmatrix})=\zeta_A(1)$. Therefore, Theorem \ref{lambda} implies that
%Since \cite[Formula (12)]{GP21}, one has
%\begin{align*}
%\mathcal{G}_{\Sigma}(f)&=\sum_{\underline{i}}\sum_{\underline{j}\leq\underline{i}}\frac{f(\theta^{q^{\underline{j}}})}{D_{\underline{j}}\ell_{\underline{i}-\underline{j}}^{q^{\underline{j}}}}\\
%&=\sum_{\underline{j}}(\sum_{\underline{i}}\frac{1}{\ell_{\underline{i}}^{q^{\underline{j}}}})\frac{f(\theta^{q^{\underline{j}}})}{D_{\underline{j}}}\\
%&=\sum_{\underline{j}}(\prod_{r\in \Sigma}(\sum_{i}\frac{1}{\ell_{i}^{q^{j_r}}}))\frac{f(\theta^{q^{\underline{j}}})}{D_{\underline{j}}}\\
%&=\sum_{\underline{j}}(\sum_{i}\frac{1}{\ell_{i}})^{\sum_{r\in \Sigma}q^{j_r}}\frac{f(\theta^{q^{\underline{j}}})}{D_{\underline{j}}}.
%\end{align*}
%Since Equality (\ref{eq: Carlitz generating function}) and the definition of $\zeta_A$, one has 
%$$\sum_{d\geq 0}\frac{1}{\ell_{d}}=\sum_{d\geq 0}S_d(1)=\zeta_A( 1).$$
%Thus 
\begin{equation}\label{Gsig}
\mathcal{G}_{\Sigma}(f)=\sum_{\underline{j}}\zeta_A( 1)^{\sum_{r\in \Sigma}q^{j_r}}\frac{f(\theta^{q^{\underline{j}}})}{D_{\underline{j}}}.
\end{equation}

Now  we will describe $\ker(\mathcal{G}_{\Sigma})$. In order to do this, let $\mathsf{M}$ be the $\mathcal{Z}_{\emptyset}(K)$-submodule of $\mathcal{Z}_{\emptyset}(K)[T_i]_{i\in\Sigma}$ generated by $\{\prod_{i\in\Sigma}T_i^{q^{j_i}-1}|\underline{0}\neq\underline{j}=(j_i)_{i\in\Sigma}\in\mathbf{N}^{|\Sigma|}\}$. We consider the  morphism $\Phi:\mathsf{M}\to \mathcal{Z}^{\text{triv}}_{\Sigma}(K)$ of $\mathcal{Z}_{\emptyset}(K)$-modules given by
 $$\Phi(\prod_{i\in\Sigma}T_i^{q^{j_i}-1})=-\zeta_A( 1)^{\sum_{i\in \Sigma}(q^{j_i}-1)}\eta_{\underline{0}}+D_{\underline{j}}\eta_{\underline{j}}$$
for all $\underline{0}\neq\underline{j}=(j_i)_{i\in\Sigma}$.
 \begin{theorem}\label{countex}The map $\Phi$ is an injection whose image is $\ker(\mathcal{G}_{\Sigma})$.
 \end{theorem}
\begin{proof}
Let $a=\sum_{\underline{0}\neq\underline{j}=(j_i)_{i\in\Sigma}}a_{\underline{j}}\prod_{i\in\Sigma}T_i^{q^{j_i}-1}\in \mathsf{M}$ with $a_{\underline{j}}=0$ for all but finitely many $\underline{j}$. Suppose that $a_{\underline{j}}\neq 0$ for some $\underline{j}$. Lemma  \ref{lem: eta} and the definition of $\Phi$
 imply that $\Phi(a)(\theta^{q^{\underline{j}}})=D_{\underline{j}}a_{\underline{j}}\neq 0$. Therefore $\Phi$ is injective. 

It remains to show that $\im(\Phi)=\ker(\mathcal{G}_{\Sigma})$. First of all, we will show that $\im(\Phi)\subset\ker(\mathcal{G}_{\Sigma})$. In other words, we need to prove that $\mathcal{G}_{\Sigma}(\Phi(a))=0$ for all $a\in \mathsf{M}$. Let $a=\sum_{\underline{0}\neq\underline{j}=(j_i)_{i\in\Sigma}}a_{\underline{j}}\prod_{i\in\Sigma}T_i^{q^{j_i}-1}\in \mathsf{M}$, where $a_{\underline{j}}=0$ for all but finitely many $\underline{j}$. Because of Equality (\ref{Gsig}), Lemma \ref{lem: eta}, and the definition of $\Phi$, one has
\begin{align*}
\mathcal{G}_{\Sigma}(\Phi(a))&=\sum_{\underline{j}}\zeta_A( 1)^{\sum_{r\in \Sigma}q^{j_r}}\frac{\Phi(a)(\theta^{q^{\underline{j}}})}{D_{\underline{j}}}\\
&=\zeta_A( 1)^{|\Sigma|}\frac{\Phi(a)(\theta^{q^{\underline{0}}})}{D_{\underline{0}}}+\sum_{\underline{0}\neq\underline{j}}\zeta_A( 1)^{\sum_{r\in \Sigma}q^{j_r}}\frac{\Phi(a)(\theta^{q^{\underline{j}}})}{D_{\underline{j}}}\\
&=-\zeta_A( 1)^{|\Sigma|}\sum_{\underline{0}\neq\underline{j}=(j_i)_{i\in\Sigma}}a_{\underline{j}}\zeta_A( 1)^{\sum_{i\in \Sigma}(q^{j_i}-1)}+\sum_{\underline{0}\neq\underline{j}}\zeta_A( 1)^{\sum_{r\in \Sigma}q^{j_r}}\frac{\Phi(a)(\theta^{q^{\underline{j}}})}{D_{\underline{j}}}\\
&=-\sum_{\underline{0}\neq\underline{j}=(j_i)_{i\in\Sigma}}a_{\underline{j}}\zeta_A( 1)^{\sum_{i\in \Sigma}q^{j_i}}+\sum_{\underline{0}\neq\underline{j}}\zeta_A( 1)^{\sum_{r\in \Sigma}q^{j_r}}\frac{\Phi(a)(\theta^{q^{\underline{j}}})}{D_{\underline{j}}}\\
&=-\sum_{\underline{0}\neq\underline{j}=(j_i)_{i\in\Sigma}}a_{\underline{j}}\zeta_A( 1)^{\sum_{i\in \Sigma}q^{j_i}}+\sum_{\underline{0}\neq\underline{j}}\zeta_A( 1)^{\sum_{r\in \Sigma}q^{j_r}}a_{\underline{j}}\\
&=0.
\end{align*}
Thus $\im(\Phi)\subset\ker(\mathcal{G}_{\Sigma})$. Inversely, if $f\in \ker(\mathcal{G}_{\Sigma})$ then $f(\theta^{q^{\underline{j}}})=0$ for all but finitely many $\underline{j}$, thus the element $$a=\sum_{\underline{j}\neq\underline{0}}\frac{f(\theta^{q^{\underline{j}}})}{D_{\underline{j}}}\prod_{i\in\Sigma}T_i^{q^{j_i}-1}$$ belongs to $\mathsf{M}.$
We will prove that $f=\Phi(a)$. Indeed, as in the proof of \cite[Theorem 6.10]{GP21}, it suffices to show that $$f(\theta^{q^{\underline{j}}})=\Phi(a)(\theta^{q^{\underline{j}}})$$
for all $\underline{j}$. Since the definition of $\Phi$ and Lemma \ref{lem: eta},  this fact is trivial if $\underline{j}\neq\underline{0}$. If $\underline{j}=0$, we use the definition of $\Phi$ and Lemma \ref{lem: eta} again to have that
\begin{align*}
\Phi(a)(\theta^{q^{\underline{0}}})&=-\sum_{\underline{j}\neq 0}\zeta_A( 1)^{\sum_{i\in \Sigma}(q^{j_i}-1)}\frac{f(\theta^{q^{\underline{j}}})}{D_{\underline{j}}}\\
&=-\zeta_A( 1)^{-|\Sigma|}\sum_{\underline{j}\neq\underline{0}}\zeta_A( 1)^{\sum_{i\in \Sigma}q^{j_i}}\frac{f(\theta^{q^{\underline{j}}})}{D_{\underline{j}}}.
\end{align*}
On the other hand, since $\mathcal{G}_{\Sigma}(f)=0$ and Equality (\ref{Gsig}), one has $$-\zeta_A( 1)^{-|\Sigma|}\sum_{\underline{j}\neq\underline{0}}\zeta_A( 1)^{\sum_{i\in \Sigma}q^{j_i}}\frac{f(\theta^{q^{\underline{j}}})}{D_{\underline{j}}}=f(\theta^{q^{\underline{0}}}).$$
Thus $f(\theta^{q^{\underline{0}}})=\Phi(a)(\theta^{q^{\underline{0}}})$ and $\im(\Phi)\supset\ker(\mathcal{G}_{\Sigma})$. Therefore we can conclude that $\im(\Phi)=\ker(\mathcal{G}_{\Sigma})$.
\end{proof}
\begin{proof}[Proof of Theorem \ref{thm: main A}]This follows immediately from Theorem \ref{countex} and the definition of the map $\Phi:\mathsf{M}\to \mathcal{Z}^{\text{triv}}_{\Sigma}(K)$.
\end{proof}
\begin{proof}[Proof of Corollary \ref{ninjective}] The claim is trivial by Theorem \ref{countex} and the fact that $\mathsf{M}\neq 0$.
\end{proof}
\begin{proof}[Proof of Theorem \ref{thm: main B}] Because of Corollary \ref{generators} and Theorem \ref{thm: main A}, $\im(\mathcal{G}_{\Sigma})$ is the ideal of $\mathcal{Z}_{\emptyset}(K)$ generated by $\mathcal{G}_{\Sigma}(\eta_{\underline{0}})$. Since Equality (\ref{Gsig}) and Lemma \ref{lem: eta}, one has $$\mathcal{G}_{\Sigma}(\eta_{\underline{0}})=\zeta_A(1)^{|\Sigma|}.$$
This finishes our proof.
\end{proof}
\begin{example}Suppose that $q>|\Sigma|>1$, let $\underline{j}\neq \underline{j}' \in \mathbf{N}^{|\Sigma|}$ such that $\underline{j}$ is a permutation of $\underline{j'}$. We set $$a=\frac{1}{D_{\underline{j}}}\prod_{i\in\Sigma}T_i^{q^{j_i}-1}-\frac{1}{D_{\underline{j}'}}\prod_{i\in\Sigma}T_i^{q^{j'_i}-1}\in \mathsf{M}\setminus \{0\}.$$
Then $0\neq f=\eta_{\underline{j}}-\eta_{\underline{j}'}=\Phi(a)\in \ker(\mathcal{G}_{\Sigma})$.

Suppose that $|\Sigma|=1$. Let $a=T_1\in \mathsf{M}\setminus\{0\}$ then $0\neq f=-\zeta_A(1)^{q-1}\eta_{0}+D_1\eta_{1}=\Phi(a)\in \ker(\mathcal{G}_{\Sigma})$.
\end{example}

\end{document}